\documentclass[reqno]{amsart}
\usepackage{amsmath,amsthm,amssymb,mathrsfs,stmaryrd,eucal,mathbbol}
\usepackage[all,cmtip]{xy}
\usepackage{hyperref}

\numberwithin{equation}{section}

\newcommand\void[1]{}

\newcommand{\C}{\mathbb{C}}

\newcommand{\R}{\mathbb{R}}
\newcommand{\Z}{\mathbb{Z}}

\newcommand{\BH}{\mathbb{H}}

\newcommand{\CA}{\mathcal{A}}
\newcommand{\CB}{\mathcal{B}}
\newcommand{\CC}{\mathcal{C}}
\newcommand{\CD}{\mathcal{D}}
\newcommand{\CE}{\mathcal{E}}

\newcommand{\CG}{\mathcal{G}}
\newcommand{\CH}{\mathcal{H}}

\newcommand{\CM}{\mathcal{M}}
\newcommand{\CN}{\mathcal{N}}

\newcommand{\CP}{\mathcal{P}}
\newcommand{\CQ}{\mathcal{Q}}
\newcommand{\CR}{\mathcal{R}}

\newcommand{\CX}{\mathcal{X}}
\newcommand{\CY}{\mathcal{Y}}

\newcommand{\FZ}{\mathfrak{Z}}

\newcommand{\Cat}{\mathrm{Cat}}
\newcommand{\Hilb}{\mathrm{Hilb}}

\newcommand{\Vect}{\mathrm{Vec}}
\newcommand{\op}{\mathrm{op}}
\newcommand{\rev}{\mathrm{rev}}

\newcommand{\TO}{\mathcal{QL}}

\newcommand{\KarCat}{\mathrm{KarCat}}
\newcommand{\Kar}{\mathrm{Kar}}

\newcommand{\sk}{\mathrm{sk}}

\newcommand{\Sep}{\mathcal{S}ep}
\newcommand{\MFus}{\mathcal{MF}us}
\newcommand{\Fus}{\mathcal{F}us}
\newcommand{\BFus}{\mathcal{BF}us}
\newcommand{\ind}{\mathrm{ind}}
\newcommand{\cl}{\mathrm{cl}}

 \DeclareMathOperator{\Hom}{Hom}

 \DeclareMathOperator{\Ker}{Ker}
 
 \DeclareMathOperator{\Img}{Im}
 \DeclareMathOperator{\Id}{Id}

 \DeclareMathOperator{\chara}{char}

 \DeclareMathOperator{\one}{\mathbf1}

 \DeclareMathOperator{\Fun}{Fun}

 \DeclareMathOperator{\LMod}{LMod}
 \DeclareMathOperator{\RMod}{RMod}
 \DeclareMathOperator{\BMod}{BMod}

 \DeclareMathOperator{\Rep}{Rep}
 \DeclareMathOperator{\sRep}{sRep}
 \DeclareMathOperator{\MExt}{MExt}
 \DeclareMathOperator{\Gal}{Gal}

\newtheorem{thm}{Theorem$^*$}[section]
\newtheorem{lem}[thm]{Lemma$^*$}
\newtheorem{lem0}[thm]{Lemma}
\newtheorem{prop}[thm]{Proposition$^*$}
\newtheorem{cor}[thm]{Corollary$^*$}

\theoremstyle{definition}
\newtheorem{defn}[thm]{Definition}

\newtheorem{exam}[thm]{Example}
\newtheorem{rem}[thm]{Remark}

\theoremstyle{remark}

\newcommand\arXiv[1]{\href{http://arxiv.org/abs/#1}{\nolinkurl{arXiv:#1}}}

\newcommand\condense{\mathrel{\,\hspace{.75ex}\joinrel\rhook\joinrel\hspace{-.75ex}\joinrel\rightarrow}}

\begin{document}

\title{Categories of quantum liquids II}
\maketitle

\begin{center}
{\large
Liang Kong$^{a,b,c}$,\,
Hao Zheng$^{d,e,f}$\,
~\footnote{Emails:
{\tt  kongl@sustech.edu.cn, haozheng@mail.tsinghua.edu.cn}}}
\\[1em]
{\small $^a$ Shenzhen Institute for Quantum Science and Engineering, \\
Southern University of Science and Technology, Shenzhen 518055, China
\\[0.5em]
$^b$ International Quantum Academy, Shenzhen 518048, China
\\[0.5em]
$^c$
Guangdong Provincial Key Laboratory of Quantum Science and Engineering, \\
Southern University of Science and Technology, Shenzhen, 518055, China
\\[0.5em]
$^d$ Institute for Applied Mathematics, Tsinghua University, Beijing 100084, China 
\\[0.5em]
$^e$ Beijing Institute of Mathematical Sciences and Applications, Beijing 101408, China 
\\[0.5em]
$^f$ Department of Mathematics, Peking University, Beijing, 100871, China }
\end{center}

\vspace{0.1cm}
\begin{abstract}
We continue to develop the theory of separable higher categories, including center functors, higher centralizers, modular extensions and group theoretical higher fusion categories. Moreover, we outline a theory of orthogonal higher categories to treat anti-unitary symmetries. Using these results we derive a systematic classification of gapped quantum liquids and predict many new SPT orders in spacetime dimension $\ge3$.
\end{abstract}

\vspace{0.3cm}
\tableofcontents

\section{Introduction}

In the first part of this series \cite{KZ22}, we developed a mathematical theory of separable higher categories (including higher $E_m$-multi-fusion categories) based on Gaiotto and Johnson-Freyd's work on condensation completion \cite{GJF19}, and we also introduced and studied higher categories of quantum liquids in all spacetime dimensions.

The goal of this work is to further develop the mathematical theory of separable higher categories and derive a systematic classification of {\em gapped} quantum liquids. 
With the aid of the powerful theory of condensation completion \cite{GJF19}, we are able to generalize a lot of results on fusion 1-categories to higher fusion categories and provide mathematical proofs of many earlier physical results on gapped quantum liquids obtained or conjectured in \cite{KWZ15, LKW18, KLWZZ20a} and their generalizations.


\smallskip

In the rest of this section, we provide some explanations of the physical origins and a brief summary of our results.

\smallskip

\textbf{Center functors.} This is the main topic of Section \ref{sec:cen-fun}.
One of the fundamental guiding principles in our study of quantum liquids is the {\em boundary-bulk relation} proposed in \cite{KWZ15,KWZ17}, which says that the bulk of a quantum liquid is the center of its boundaries. This result generalizes and unifies the boundary-bulk relations in 1+1D rational conformal field theories (RCFT)\cite{FFRS08,KR08,Dav10} and those in 2+1D topological orders \cite{KK12} (in a lattice model approach) and \cite{FSV14,Kon14} (in a model-independent approach). 
However, as observed in \cite{DKR15,KWZ15}, this boundary-bulk relation is only the first layer of a hierarchical structure, which can be obtained by including higher codimensional defects. As a consequence, a complete statement of the boundary-bulk relation can be formulated as the functoriality of the $E_1$-center or the so-called $E_1$-center functor. This $E_1$-center functor has been rigorously constructed for a few special families of quantum liquids: (1) for all 1+1D RCFT’s with a fixed chiral symmetry \cite{DKR15} and for all 1+1D RCFT’s with arbitrary chiral symmetries \cite{KYZ21}; (2) for 2+1D non-chiral topological orders with only gapped boundaries \cite{KZ18} and for all 2+1D topological orders with all (potentially gapless) boundaries \cite{KZ21}. These rigorous results only involve monoidal 1-categories or algebras in them. 

In this work, we prove the functoriality of the $E_1$-center for higher categories (see Theorem \ref{thm:z1}). Although the proofs for the 1-category cases \cite{KZ18,KYZ21,KZ21} are very technical, it turns out that the proof for higher categories are surprisingly simple. The reason behind this miracle is that the theory of separable $n$-categories allows us to reduce the functoriality of the $E_1$-center to that of the $E_0$-center (see Theorem \ref{thm:z0}), which is much easier to prove. These center functors provide powerful tools that are useful in other studies, for example, in the study of modular extensions. 

\smallskip

\textbf{Modular extensions.}\footnote{Modular extensions are also called nondegenerate extensions in the literature. We follow M\"uger's terminology. In fact, a nondegenerate braided fusion category admits a canonical modular structure in the unitary settings.}
This is the main subject of Section \ref{sec:cen-mext}. 
The mathematical notion of a (minimal) modular extension\footnote{In this work, a modular extension means a minimal modular extension.} of a braided fusion 1-category was introduced by M\"uger in \cite{Mu03b}, but was largely unnoticed until the discovery of its physical meaning as a gauging process of the finite internal symmetry in a 2+1D SPT/SET order \cite{LKW15,LKW17,BGHN+17}. Extensive research on this subject followed (see, for example, \cite{GVR17,VR19,DN20,Nik22,OY23,JFR24} and references therein). Based on a different realization of the gauging process via the boundary-bulk relation, a one-to-one correspondence between modular extensions and braided equivalences of $E_1$-centers was obtained on a physical level of rigor and was generalized to higher categories \cite{KLWZZ20a}. 

Roughly speaking, the theory of modular extensions tells the story about a pair of $E_m$-fusion $n$-categories centralizing each other in a nondegenerate one. We call such a structure a perfect pairing (see Definition \ref{defn:pairing}). In order to define the notion of a modular extension properly for higher categories, we first develop the mathematical theory of $E_m$-centralizers in Subsection \ref{sec:centralizers} (largely parallel to that of $E_m$-centers in \cite{KZ22}) and that of perfect pairings in Subsection \ref{sec:pairing}. Similar to the proof of the functoriality of centers, we reduce the original problem of modular extensions of an $E_2$-fusion $n$-category (see Definition \ref{defn:e2-mext}) to the $E_1$-settings (see Definition \ref{defn:e1-mext}) and the $E_0$-settings (see Definition \ref{defn:e0-mext}). Then we establish one-to-one correspondence between modular extensions and equivalences of centers (see Theorem \ref{thm:mod-ext-e0} and Corollary \ref{cor:mod-ext-e1}, Corollary \ref{cor:mod-ext-e2}).

\smallskip
\textbf{Group-theoretical fusion $n$-categories.} This is the subject of Section \ref{sec:group-cat}. Group-theoretical fusion 1-categories are a family of well studied fusion 1-categories and play a crucial role in the classification of low-dimensional gapped quantum liquids. Moreover, higher groups are main examples of higher symmetries, so their representation theory are among the most important problems in the theory of quantum liquids. In Subsection \ref{rep-high-group}, we define the representation $n$-category of a finite higher group and prove that it is indeed a symmetric fusion $n$-category (see Theorem \ref{thm:nrepg-fus}). In Subsection \ref{sec:point-fus}, we study the so-called pointed fusion $n$-categories which is a natural generation of pointed fusion 1-categories. These results are used in the classification of gapped quantum liquids with group symmetries.

\smallskip
\textbf{Orthogonal higher categories.} This is the subject of Section \ref{sec:orth-cat}. The motivation is to set up a categorical foundation for the time-reversal symmetry. Note that the time-reversal symmetry is a kind of anti-unitary symmetry and can be naturally described by the Galois group $\Gal(\C/\R)$. We observe that the twisted group algebra of $\Gal(\C/\R)$ is Morita equivalent to $\R$ (see Example \ref{exam:rep-Gpi}(2)). Therefore, the time-reversal symmetry demands a theory of separable $n$-categories based on real Hilbert spaces instead of complex Hilbert spaces. This leads to the notion of an orthogonal $n$-category (see Definition \ref{defn:orth-cat}) which is a real analogue of a unitary $n$-category. 

The major issue of orthogonal $n$-categories is that the field $\R$ is not separably closed. Fortunately, most results on unitary $n$-categories have counterparts for orthogonal $n$-categories. There are also interesting new features arising in the orthogonal settings that are briefly discussed in this part of work. This approach to the time-reversal symmetry is justified by the examples of time-reversal SPT's presented in Subsection \ref{sec:anti-unitary-symm} (see Example \ref{exam:spt-Z2T}).

\smallskip
\textbf{Classification of gapped quantum liquids.} This is the task of Section \ref{sec:class-gql} and relies on the mathematical results established in the previous sections. Recall that gapped quantum liquids include topological orders, SPT/SET orders and symmetry-breaking orders. In Subsection \ref{sec:unitary-symm}, we give a precise formulation of and slightly generalize the classification result of gapped quantum liquids presented in \cite{KLWZZ20a} and discuss some examples. 

Then we consider gapped quantum liquids with spacetime symmetries (including the time-reversal symmetry) in the next two subsections. Indeed, such phases do not belong to the higher category $\TO^n$ introduced in \cite{KZ22} because spacetime symmetries are not local and because a domain wall between such phases hardly makes sense. We refer to them as {\em crystalline quantum liquids} and reduce their classification problem to quantum liquids without spacetime symmetries. In Subsection \ref{sec:anti-unitary-symm}, we deal mainly with the time-reversal symmetry. Surprisingly, we find new time-reversal SPT's in spacetime dimension $\ge3$ (see Example \ref{exam:spt-Z2T}, Example \ref{exam:newspt-Z4T} and Example \ref{exam:newspt-Z2fxZ2T}). In Subsection \ref{sec:ext-symm}, we deal with general finite spacetime symmetries and provide a categorical proof to the crystalline equivalence principle \cite{TE18} which claims that space symmetries can be dualized to internal symmetries.

\medskip

We assume that readers are familiar with the notions of a multi-fusion $n$-category \cite{JF22} and a separable $n$-category \cite{KZ22} and the notations and the results in \cite{KZ22}. 

We work on the $\C$-linear settings unless stated otherwise as in \cite{KZ22}. All the results have a $*$-version. The results also apply to the $k$-linear settings, where $k$ is an arbitrary field, with some minor exceptions that are clarified at the end of each subsection.

\medskip


\noindent\textbf{Disclaimer}: Our work is based on the theory of condensation completion for higher categories \cite{GJF19}, a theory which is mathematically incomplete because no concrete model of a weak $n$-category is chosen and the yet-to-be-developed theory of (co)limits is assumed. Implementing the technical details in a specific model will be a substantial task. However, this approach, as explained in \cite{JF22}, allows us to proceed to study problems that are not explicitly dependent on the higher coherence data. 

We use Theorem$^*$, Proposition$^*$, etc. to remind readers that the results are proven at a physical level of rigor, depending more or less on unproven assertions about higher category theory.

\medskip
\noindent{\bf Acknowledgments}:
We thank Zheng-Cheng Gu, Tian Lan, Dmitri Nikshych, Yin Tian, Yilong Wang, Xiao-Gang Wen, Yong-Shi Wu, Hao Xu, Zhi-Hao Zhang for helpful discussions. We appreciate the constructive comments and suggestions of the referee that have helped us improve the quality of our paper.
LK is supported by NSFC under Grant No. 11971219 and Guangdong Basic and Applied Basic Research Foundation under Grant No. 2020B1515120100 and Guangdong Provincial Key Laboratory (Grant No.2019B121203002). HZ is  supported by NSFC under Grant No. 11871078 and by Startup Grant of Tsinghua University and BIMSA.



\section{Elements of separable $n$-categories} \label{sec:elements}

We prove some preliminary results of separable $n$-categories in this section for later use. 

Recall that a $\C$-linear $n$-category is {\em separable} if it lies in the essential image of the embedding $(n+1)\Vect \hookrightarrow \Cat^\C_n$ \cite[Definition 3.3]{KZ22}. 

A {\em (separable) $E_m$-multi-fusion $n$-category} is a condensation-complete $\C$-linear $E_m$-monoidal $n$-category $\CA$ such that $\Sigma^m\CA$ is separable \cite[Definition 3.29]{KZ22} (see also \cite{JF22}). An $E_m$-multi-fusion $n$-category with a simple tensor unit is also referred to as an {\em $E_m$-fusion $n$-category}. We say that an $E_m$-multi-fusion $n$-category $\CA$ is {\em indecomposable} if $\Sigma^m\CA$ is indecomposable; {\em connected} if the underlying separable $n$-category of $\CA$ is indecomposable.
A {\em separable module} over a multi-fusion $n$-category is a module in $(n+1)\Vect$ \cite[Definition 3.19]{KZ22}.

\begin{prop} \label{prop:fus-nonzero}
Let $\CA$ be a fusion $n$-category. Then $X\otimes Y$ is nonzero for any two nonzero objects $X,Y\in\CA$.
\end{prop}
\begin{proof}
Since $\one_\CA$ is simple and since the counit map $v:X^L\otimes X\to\one_\CA$ is nonzero, $v$ extends to a condensation by \cite[Proposition 3.11]{KZ22}. Similarly, the counit map $Y\otimes Y^R\to\one_\CA$ extends to a condensation. Therefore, $\one_\CA$ is a condensate of $(X^L\otimes X)\otimes(Y\otimes Y^R)$. This implies that $X\otimes Y$ is nonzero.
\end{proof}

\begin{prop} \label{prop:morita-braid}
The following conditions are equivalent for $\CA,\CB\in E_1\KarCat^\C_n$:
\begin{enumerate}
\item There is a $\C$-linear equivalence $\RMod_\CA(\KarCat^\C_n) \simeq \RMod_\CB(\KarCat^\C_n)$.
\item There is a $\C$-linear equivalence $\Sigma\CA\simeq\Sigma\CB$.
\item There is an invertible $\CA$-$\CB$-bimodule in $\KarCat^\C_n$ (that is, $\CA$ and $\CB$ are Morita equivalent).
\end{enumerate}
Moreover, if these conditions are satisfied, then there is a $\C$-linear monoidal equivalence $\FZ_1(\CA)\simeq\FZ_1(\CB)$.
\end{prop}
\begin{proof}
$(1)\Rightarrow(2)$ The full subcategory $\Sigma\CA$ of $\RMod_\CA(\KarCat^\C_n)$ consists of tiny objects.
$(2)\Rightarrow(3)$ A $\C$-linear equivalence $F:\Sigma\CA\to\Sigma\CB$ gives an invertible $\CA$-$\CB$-bimodule $F(\CA)$.
$(3)\Rightarrow(1)$ An invertible $\CA$-$\CB$-bimodule $\CM$ induces a $\C$-linear equivalence $-\boxtimes_\CA \CM: \RMod_\CA(\KarCat^\C_n) \to \RMod_\CB(\KarCat^\C_n)$.
Moreover, since $\FZ_1(\CA)=\Omega\FZ_0(\Sigma\CA)$, a $\C$-linear equivalence $\Sigma\CA\simeq\Sigma\CB$ induces a $\C$-linear monoidal equivalence $\FZ_1(\CA)\simeq\FZ_1(\CB)$.
\end{proof}


\begin{prop} \label{prop:dual-bimod}
Let $\CA,\CB$ be multi-fusion $n$-categories and $\CM$ be a separable $\CA$-$\CB$-bimodule. $(1)$ The $\CB$-$\CA$-bimodule $\Fun_{\CB^\rev}(\CM,\CB)$ is left dual to $\CM$. $(2)$ The $\CB$-$\CA$-bimodule $\Fun_\CA(\CM,\CA)$ is right dual to $\CM$.
\end{prop}
\begin{proof}
(2) is obtained from (1) by viewing $\CM$ as a $\CB^\rev$-$\CA^\rev$-bimodule. To establish (1) we need to show that $\Fun_\CA(\CP,\CM\boxtimes_\CB-) \simeq \Fun_\CB(\Fun_{\CB^\rev}(\CM,\CB)\boxtimes_\CA\CP,-)$ for $\CP\in\Sigma\CA^\rev$. The equivalence holds for $\CP=\CA$ as $\Fun(n\Vect,\CM\boxtimes_\CB-) \simeq \Fun_\CB(\Fun_{\CB^\rev}(\CM,\CB),-)$ by \cite[Theorem 3.1(3)]{KZ22}, hence holds for any condensate $\CP$ of $\CA$.
\end{proof}

\begin{prop} \label{prop:tensor-nonzero}
Let $\CA$ be a multi-fusion $n$-category, $\CM$ be a separable right $\CA$-module and $\CN$ be a separable left $\CA$-module. $(1)$ We have a $\C$-linear equivalence 
$$\CM\boxtimes_\CA\CN \simeq \Fun_\CA(\CM^L,\CN), \quad X\boxtimes_\CA Y \mapsto (f\mapsto f(X)\otimes Y),$$ 
where $\CM^L = \Fun_{\CA^\rev}(\CM,\CA)$. $(2)$ If $\CA$ is a fusion $n$-category, then $X\boxtimes_\CA Y$ is nonzero for nonzero $X\in\CM$ and $Y\in\CN$.
\end{prop}
\begin{proof}
(1) By \cite[Theorem 3.1(3)]{KZ22}, $\CM\boxtimes_\CA\CN \simeq \Fun(n\Vect,\CM\boxtimes_\CA\CN) \simeq \Fun_\CA(\CM^L,\CN)$ that maps $X\boxtimes_\CA Y$ to the composition $\CM^L \xrightarrow{\Id_{\CM^L}\boxtimes X\boxtimes_\CA Y} \CM^L\boxtimes\CM\boxtimes_\CA\CN \to \CA\boxtimes_\CA\CN \simeq \CN$.
(2) In particular, we have $\CM\simeq\Fun_\CA(\CM^L,\CA)$, $X\mapsto(f\mapsto f(X))$. Then there exists $f\in\CM^L$ such that $f(X)$ is nonzero. Then $\one_\CA\otimes Y$ is a condensate of $f(X)^L\otimes f(X)\otimes Y$ hence $f(X)\otimes Y$ is nonzero. Therefore, $X\boxtimes_\CA Y$ is nonzero.
\end{proof}

Recall that an $E_0$-monoidal $n$-category consists of a pair $(\CC,X)$ where $\CC$ is an $n$-category and $X$ is an object of $\CC$. In particular, if $\CA$ is a monoidal $n$-category then $(\Sigma\CA,\CA)$ defines an $E_0$-monoidal $(n+1)$-category.

\begin{prop} \label{prop:fun-inv}
Let $\CA$ be an indecomposable multi-fusion $n$-category and let $\CB=\Fun_\CA(\CM,\CM)$ where $\CM$ is a nonzero separable left $\CA$-module. $(1)$ The $\CA$-$\CB^\rev$-bimodule $\CM$ is invertible. $(2)$ The canonical $\C$-linear monoidal functor $\CA\to\Fun_\CB(\CM,\CM)$ is invertible.
\end{prop}
\begin{proof}
Consider the indecomposable multi-fusion $n$-category 
$$\Fun_\CA(\CA\oplus\CM,\CA\oplus\CM) \simeq \begin{pmatrix} \CA^\rev & \Fun_\CA(\CM,\CA) \\ \CM & \CB \end{pmatrix}.$$
The composition $\Sigma\CA \simeq \Sigma\Fun_\CA(\CA\oplus\CM,\CA\oplus\CM)^\rev \simeq \Sigma\CB^\rev$ maps $\CA \mapsto \CA\oplus\CM \mapsto \CM$. We obtain (1). Moreover, $\Fun_\CB(\CM,\CM) \simeq \Omega(\Sigma\CB^\rev,\CM) \simeq \Omega(\Sigma\CA,\CA) \simeq \CA$. We obtain (2).
\end{proof}

\begin{prop} \label{prop:fun-tensor}
Let $\CA,\CB$ be multi-fusion $n$-categories. Then 
$$\Fun_\CA(\CM,\CN)\boxtimes_\CB\CP \simeq \Fun_\CA(\CM,\CN\boxtimes_\CB\CP)$$ 
for separable left $\CA$-module $\CM$, $\CA$-$\CB$-bimodule $\CN$ and left $\CB$-module $\CP$.
\end{prop}
\begin{proof}
By \cite[Theorem 3.1(2)]{KZ22}, $\Fun_\CA(\CM,-)$ preserves colimits.
\end{proof}

\begin{defn}
We say that $\CA\in E_m\KarCat^\C_n$ is {\em trivial} if the canonical $\C$-linear $E_m$-monoidal functor $n\Vect \to \CA$ is invertible.
\end{defn}

\begin{lem0} \label{lem:trivial-alg}
Let $\CC$ be a monoidal 1-category. An algebra $A$ in $\CC$ is trivial (i.e. the unit $u_A:\one\to A$ of $A$ is an isomorphism) if and only if the underlying object of $A$ is invertible.\footnote{We thank the referee for pointing out this lemma.}
\end{lem0}
\begin{proof}
One direction is trivial. Assume the underlying object of $A$ is invertible. Then $A$ is dualizable with invertible unit map $u:\one\to A\otimes A^L$ and counit map $v:A^L\otimes A\to\one$. Let $w$ denote the composition $$A \simeq A\otimes\one \xrightarrow{A\otimes u} A\otimes A\otimes A^L \xrightarrow{m_A\otimes A^L} A\otimes A^L \xrightarrow{u^{-1}} \one$$ where $m_A$ is the product of $A$. Then $w\circ u_A = u^{-1}\circ u = \Id_{\one}$ by the unity property of $A$. Moreover, $u_A$ equals to the composition $$\one \simeq \one\otimes\one \xrightarrow{u\otimes u_A} A\otimes A^L\otimes A \xrightarrow{A\otimes v} A.$$ Thus $u_A\circ w = \Id_A$ by the unity property of $A$ again. Therefore, $u_A=w^{-1}$.
\end{proof}

\begin{prop} \label{prop:fus-trivial}
For $m\ge1$, $\CA\in E_m\KarCat^\C_n$ is trivial if and only if the underlying $\C$-linear $n$-category of $\CA$ is invertible (as an object of $\KarCat^\C_n$).
\end{prop}
\begin{proof}
Since the forgetful functor $E_m\KarCat^\C_n\to\KarCat^\C_n$ is conservative (c.f. \cite[Lemma 3.2.2.6]{Lur14}), $\CA$ is trivial if and only if the canonical $\C$-linear functor $n\Vect \to \CA$ is invertible.
Applying Lemma \ref{lem:trivial-alg} to the homotopy 1-category of $\KarCat^\C_n$, we conclude the proposition.
\end{proof}


\begin{cor} \label{cor:inv-trivial}
A separable $n$-category $\CC$ is invertible (as an object of $\KarCat^\C_n$) if and only if $\FZ_0(\CC)$ is trivial.
\end{cor}
\begin{proof}
We have a $\C$-linear equivalence $\FZ_0(\CC)\simeq\CC^\op\boxtimes\CC$. Therefore, $\CC$ is invertible if and only if $\FZ_0(\CC)\simeq n\Vect$.
\end{proof}

\begin{prop} \label{prop:sep-idem}
Let $\CC$ be a nonzero separable $n$-category. If $\CC\boxtimes\CC\simeq\CC$, then $\CC\simeq n\Vect$.
\end{prop}
\begin{proof}
The claim is clear for $n=0$. For $n\ge1$, we have $\FZ_0(\CC)\boxtimes\FZ_0(\CC) \simeq \FZ_0(\CC\boxtimes\CC) \simeq \FZ_0(\CC)$. By the inductive hypothesis, we have $\Omega\FZ_0(\CC)\simeq(n-1)\Vect$. The equivalence $\FZ_0(\CC)\boxtimes\FZ_0(\CC) \simeq \FZ_0(\CC)$ implies that $\FZ_0(\CC)$ is connected so that $\FZ_0(\CC) \simeq n\Vect$. Hence $\FZ_0(\CC)$ is trivial by Proposition \ref{prop:fus-trivial}. Therefore, $\CC$ is invertible by Corollary \ref{cor:inv-trivial} hence equivalent to $n\Vect$.
\end{proof}

\begin{cor} \label{cor:ret-nvec}
Let $\CC$ be a nonzero $\C$-linear $n$-category. If $\C$-linear functors $i:\CC\to n\Vect$ and $r:n\Vect\to\CC$ exhibit $\CC$ as a retract of $n\Vect$, i.e. $r\circ i \simeq \Id_\CC$, then $i$ and $r$ are inverse to each other.
\end{cor}
\begin{proof}
Let $e=i\circ r$ so that $e\circ e\simeq e$. Since $\CC$ is nonzero, $e$ is nonzero. By regarding $e$ as an object of $n\Vect$, we conclude from Proposition \ref{prop:sep-idem} that $e\simeq\Id_{n\Vect}$.
\end{proof}

Recall that we use $\CC^{\op k}$ to denote the $n$-category obtained by reversing all the $k$-morphisms for an $n$-category $\CC$ and $k>0$. Moreover, we use $\CA^{\op k}$ to denote the $E_m$-monoidal $n$-category obtained by reversing all the $k$-morphisms for an $E_m$-monoidal $n$-category $\CA$ and $k>-m$, i.e. $B^m(\CA^{\op k}) = (B^m\CA)^{\op(k+m)}$. 

\begin{defn}
A {\em condensation-complete $\C$-linear $E_m$-monoidal $n$-category over $\CA\in E_{m+1}\KarCat^\C_n$} is an object $\CM\in E_m\KarCat^\C_n$ equipped with a $\C$-linear $E_{m+1}$-monoidal functor $\psi_\CM:\CA\to\FZ_m(\CM)$. We say that two condensation-complete $\C$-linear $E_m$-monoidal $n$-categories $\CM,\CN$ over $\CA$ are {\em equivalent} if there is a $\C$-linear $E_m$-monoidal equivalence $\CM\simeq\CN$ such that the composition $\CA\xrightarrow{\psi_\CM}\FZ_m(\CM)\simeq\FZ_m(\CN)$ is equivalent to $\psi_\CN$.
\end{defn}

\begin{rem}
Let $\CA\in E_{m+1}\KarCat^\C_n$ and $\CM\in E_m\KarCat^\C_n$. Note that $\CM$ admits an associative left action of $\FZ_m(\CM)$ as well as an associative right action of $\FZ_m(\CM)^{\op(-m)}$ in $E_m\KarCat^\C_n$. Promoting $\CM$ to a $\C$-linear $E_m$-monoidal $n$-category over $\CA$ is equivalent to endowing $\CM$ with an associative left action of $\CA$ in $E_m\KarCat^\C_n$.
For example, the evident associative left action of $\CA$ on itself induces a $\C$-linear $E_{m+1}$-monoidal functor $\CA\to\FZ_m(\CA)$, promoting $\CA$ to a $\C$-linear $E_m$-monoidal $n$-category over itself.
\end{rem}

\begin{rem}
Let $\CC$ be a separable $n$-category. The natural equivalence $\Hom_\CC(X,F(Y)) \simeq \Hom_\CC(F^L(X),Y)$ for $X,Y\in\CC$ induces a $\C$-linear monoidal equivalence
$$\Fun(\CC,\CC)^\rev \simeq \Fun(\CC^\op,\CC^\op), \quad F\mapsto(F^L)^\op.$$
This yields a canonical $\C$-linear $E_{m+1}$-monoidal equivalence by \cite[Theorem 3.43]{KZ22}
$$\FZ_m(\CA)^{\op(-m)} \simeq \FZ_m(\CA^{\op(1-m)})$$
for every $E_m$-multi-fusion $n$-category $\CA$. In this way, if $\CA$ is an $E_{m+1}$-multi-fusion $n$-category, then $\CA^{\op(1-m)}$ is an $E_m$-multi-fusion $n$-category over $\CA^{\op(-m)}$.
\end{rem}

Let $K/k$ be a field extension. 
The inclusion $\iota:k\hookrightarrow K$ induces a $k$-linear symmetric monoidal functor 
$$\Sigma^{n+1}\iota: \Sigma^{n+1}k\to\Sigma^{n+1}K, \quad \CC\mapsto\CC^K:=\CC\boxtimes_k\Sigma^n K.$$
That is, if $\CC$ is a separable $n$-category over $k$ (i.e. an object of $\Sigma^{n+1}k$) then $\CC^K$ is a separable $n$-category over $K$ (i.e. an object of $\Sigma^{n+1}K$) and, moreover, $\CC^K\boxtimes_K\CD^K \simeq (\CC\boxtimes_k\CD)^K$. We refer to $\CC^K$ as the {\em extension} of $\CC$ by $K$. 

\begin{rem}
Field extensions are compatible with many categorical constructions. For example, $\Fun_K(\CC^K,\CD^K) \simeq \Fun_k(\CC,\CD)^K$ for separable $n$-categories $\CC,\CD$ over $k$, $\Sigma(\CA^K) \simeq (\Sigma\CA)^K$, $\Omega(\CA^K) \simeq (\Omega\CA)^K$ and $\FZ_m(\CA^K) \simeq \FZ_m(\CA)^K$ for an $E_m$-multi-fusion $n$-category $\CA$ over $k$. If $K/k$ is a separable finite extension, then $K$ is a symmetric fusion 0-category over $k$ therefore $\Sigma^{n+1}K$ is a symmetric fusion $(n+1)$-category over $k$; the restriction functor $\Sigma^{n+1}K\to\Sigma^{n+1}k$ is right adjoint to $\Sigma^{n+1}\iota$.
\end{rem}

\begin{prop}
Let $K/k$ be a field extension. $(1)$ A fusion $n$-category $\CA$ over $k$ is trivial if and only if $\CA^K$ is trivial. $(2)$ A separable $n$-category $\CC$ over $k$ is invertible if and only if $\CC^K$ is invertible. 
\end{prop}
\begin{proof}
(1) The claim is clearly true for $n=0$. For $n\ge1$, if $\CA^K$ is trivial then $\Omega\CA^K$ is trivial hence $\Omega\CA$ is trivial by the inductive hypothesis hence $\CA$ is trivial. 

(2) $\CC$ is invertible if and only if $\FZ_0(\CC)$ is trivial by Corollary \ref{cor:inv-trivial}, if and only if $\FZ_0(\CC^K)$ is trivial by (1), if and only if $\CC^K$ is invertible. 
\end{proof}

\begin{rem}
Proposition \ref{prop:dual-bimod} for $n=1$ is due to \cite[Proposition 3.2.1]{DSPS20}.
Proposition \ref{prop:tensor-nonzero}(1) for $n=1$ is due to \cite[Proposition 3.5]{ENO10}, \cite[Theorem 3.20]{Gr10} and \cite[Proposition 2.4.10]{DSPS20}.
Proposition \ref{prop:fun-inv} for $n=1$ is due to \cite[Theorem 3.27 and 3.31]{EO04}.
\end{rem}

\section{Center functors} \label{sec:cen-fun}

In this section, we formulate and prove the functoriality of $E_0$-centers for indecomposable separable $n$-categories as well as $E_1$-centers for indecomposable multi-fusion $n$-categories. The center functors offer a conceptual framework to explore various aspects of higher fusion categories, and the related concepts will be recurrent throughout the remainder of this paper. 
The strategy of this section is to establish the fairly simple case of $E_0$-center functors first, and then deduce the highly nontrivial case of $E_1$-center functors by applying the looping-delooping construction.

\subsection{$E_0$-center functors} 

\begin{defn}
Let $\CA,\CB$ be indecomposable multi-fusion $n$-categories. We say that a separable $\CA$-$\CB$-bimodule $\CM$ is {\em closed} if the canonical $\C$-linear monoidal functor $\CA\boxtimes\CB^\rev\to\Fun(\CM,\CM)$ is invertible. We say that $\CA$ is {\em nondegenerate} if the $\CA$-$\CA$-bimodule $\CA$ is closed. 
\end{defn}

\begin{exam} \label{exam:closed-bimod}
Let $\CC,\CD$ be nonzero separable $n$-categories. Then the $\Fun(\CD,\CD)$-$\Fun(\CC,\CC)$-bimodule $\Fun(\CC,\CD)$ is equivalent to $\CC^\op\boxtimes\CD$ hence is closed. In particular, $\Fun(\CC,\CC)$ is nondegenerate.
\end{exam}

\begin{rem}
The concepts of closeness and nondegeneracy correspond to the anomaly-free condition in physics. Specifically, a nondegenerate multi-fusion $n$-category represents an anomaly-free topological order, while a closed separable bimodule represents an anomaly-free domain wall. Proposition \ref{prop:fus-nondeg} below captures the idea that a topological order is anomaly-free if and only if its bulk is trivial \cite{KWZ15,KWZ17}.
\end{rem}

\begin{prop} \label{prop:fus-nondeg}
The following conditions are equivalent for a multi-fusion $n$-category $\CA$:
\begin{enumerate}
\item $\CA$ is nondegenerate.
\item $\Sigma\CA$ is an invertible separable $(n+1)$-category.
\item $\FZ_0(\Sigma\CA)$ is trivial.
\item $\FZ_1(\CA)$ is trivial.
\end{enumerate}
\end{prop}

\begin{proof}
$(2)\Rightarrow(1)$ The equivalence $\Sigma\CA\boxtimes\Sigma(\CA^\rev) \simeq (n+1)\Vect$ maps $\CA\boxtimes\CA$ to $\Hom_{\Sigma\CA}(\CA,\CA)\simeq\CA$, inducing the desired equivalence $\CA\boxtimes\CA^\rev \simeq \Fun(\CA,\CA)$.

$(3)\Rightarrow(2)$ Since $\Sigma\CA\boxtimes(\Sigma\CA)^\op \simeq \FZ_0(\Sigma\CA) \simeq (n+1)\Vect$, $\Sigma\CA$ is invertible.

$(4)\Rightarrow(3)$ Since $\CA$ is indecomposable, $\FZ_0(\Sigma\CA) = \Sigma\FZ_1(\CA)$ by \cite[Lemma 3.50]{KZ22}.

$(1)\Rightarrow(4)$ Since $\CA\boxtimes\CA^\rev \simeq \FZ_0(\CA)$ by assumption, we have $\FZ_1(\CA)\boxtimes\FZ_1(\CA^\rev) \simeq \FZ_1(\FZ_0(\CA))$ which is trivial by \cite[Proposition 3.51(2)]{KZ22}. Therefore, $\FZ_1(\CA)$ is trivial by Proposition \ref{prop:fus-trivial}.
\end{proof}

\begin{defn}
We say that a nondegenerate multi-fusion $n$-category is {\em chiral} if it is not the $E_0$-center of a separable $n$-category.
\end{defn}

\begin{prop} \label{prop:chiral-mfc}
A nondegenerate multi-fusion $n$-category $\CA$ is chiral if and only if $\Sigma\CA$ is a nontrivial invertible separable $(n+1)$-category.
\end{prop}
\begin{proof}
If $(\Sigma\CA,\CA) \simeq ((n+1)\Vect,\CC)$, then $\CA\simeq\Fun(\CC,\CC)$, i.e. $\CA$ is non-chiral.
Conversely, if $\CA$ is non-chiral, say, $\CA\simeq\Fun(\CC,\CC)$ for a nonzero separable $n$-category $\CC$, then $\CA\simeq\Omega((n+1)\Vect,\CC)$ so that $\Sigma\CA \simeq (n+1)\Vect$. 
\end{proof}

\begin{cor}
Let $\CA,\CB$ be nondegenerate multi-fusion $n$-categories. Then $\CA$ is Morita equivalent to $\CB$ if and only if $\CA\boxtimes\CB^\rev$ is non-chiral.
\end{cor}
\begin{proof}
$\CA$ is Morita equivalent to $\CB$ if and only if $\Sigma\CA\simeq\Sigma\CB$ by Proposition \ref{prop:morita-braid}, if and only if $\Sigma\CA\boxtimes\Sigma(\CB^\rev)\simeq(n+1)\Vect$.
\end{proof}

\begin{cor}
The construction $\CA\mapsto\Sigma\CA$ gives a one-to-one correspondence between the following two types of data:
\begin{enumerate}
\renewcommand{\labelenumi}{$(\alph{enumi})$}
\item Morita equivalence classes of nondegenerate multi-fusion $n$-categories.
\item Equivalence classes of invertible separable $(n+1)$-categories.
\end{enumerate}
\end{cor}

\begin{prop} \label{prop:close-bimod1}
Let $\CA,\CB$ be nondegenerate multi-fusion $n$-categories. An $\CA$-$\CB$-bimodule $\CM$ be a closed if and only if the induced $\C$-linear functor $-\boxtimes_\CA\CM:\Sigma\CA\to\Sigma\CB$ is invertible.
\end{prop}
\begin{proof}
An $\CA$-$\CB$-bimodule $\CM$ is closed if and only if the $n\Vect$-$\CA^\rev\boxtimes\CB$-module $\CM$ is closed, if and only if $-\boxtimes\CM:(n+1)\Vect\to\Sigma\CA^\rev\boxtimes\Sigma\CB$ is invertible, if and only if $-\boxtimes_\CA\CM:\Sigma\CA\to\Sigma\CB$ is invertible.
\end{proof}

\begin{cor} \label{cor:closed-bimod}
Let $\CA,\CB,\CC$ be nondegenerate multi-fusion $n$-categories and $\CM$ be a closed $\CB$-$\CA$-bimodule, $\CN$ be a closed $\CC$-$\CB$-bimodule. Then the $\CC$-$\CA$-bimodule $\CN\boxtimes_\CB\CM$ is also closed.
\end{cor}

Let $\Sep^\ind_n$ denote the symmetric monoidal 1-category where an object is an indecomposable separable $n$-category and a morphism is an equivalence class of nonzero $\C$-linear functors. The tensor product is given by $\boxtimes$. By \cite[Corollary 3.12(1)]{KZ22}, the composition law is well-defined because an object of $\Sep^\ind_n$ is a simple object of $(n+1)\Vect$.

Let $\Fus^\cl_n$ denote the symmetric monoidal 1-category where an object is a nondegenerate fusion $n$-category $\CA$ and a morphism $\CA\to\CB$ is an equivalence class of pairs $(\CM,X)$ where $\CM$ is a closed $\CB$-$\CA$-bimodule and $X$ is a nonzero object of $\CM$. Two pairs $(\CM,X)$ and $(\CN,Y)$ are equivalent if there is a $\C$-linear bimodule equivalence $F:\CM\to\CN$ such that $F(X)\simeq Y$. The tensor product is given by $\boxtimes$. The composition law is given by relative tensor product of bimodules.

\begin{prop}
The composition law of $\Fus^\cl_n$ is well-defined.
\end{prop}

\begin{proof}
Let $(\CM,X):\CA\to\CB$ and $(\CN,Y):\CB\to\CC$ be two morphisms of $\Fus^\cl_n$. By Corollary \ref{cor:closed-bimod}, the $\CC$-$\CA$-bimodule $\CN\boxtimes_\CB\CM$ is closed. By Proposition \ref{prop:tensor-nonzero}(2), $Y\boxtimes_\CB X$ is nonzero. 
\end{proof}

\begin{thm} \label{thm:z0}
The assignment $\CC\mapsto\Fun(\CC,\CC)$ and $(\CC\xrightarrow{F}\CD)\mapsto(\Fun(\CC,\CD),F)$ defines a symmetric monoidal functor
$$\FZ_0: \Sep^\ind_n \to \Fus^\cl_n.$$
\end{thm}
\begin{proof}
By Example \ref{exam:closed-bimod}, the assignment is well-defined on objects and morphisms. 
Applying \cite[Proposition 3.28(2)]{KZ22} to the indecomposable multi-fusion $n$-category $\Fun(\CC\oplus\CD\oplus\CE,\CC\oplus\CD\oplus\CE)$, we obtain an equivalence $\Fun(\CD,\CE)\boxtimes_{\Fun(\CD,\CD)}\Fun(\CC,\CD) \simeq \Fun(\CC,\CE)$, $G\boxtimes_{\Fun(\CD,\CD)}F \mapsto G\circ F$. This shows that the assignment preserves composition of morphisms.
\end{proof}

\begin{prop} \label{prop:inv-z0}
The following conditions are equivalent:
\begin{enumerate}
\item There is no nontrivial invertible separable $n$-category.
\item $\FZ_0: \Sep^\ind_{n-1} \to \Fus^\cl_{n-1}$ is essentially surjective.
\item $\FZ_0: \Sep^\ind_n \to \Fus^\cl_n$ is full.
\item $\FZ_0: \Sep^\ind_{n+1} \to \Fus^\cl_{n+1}$ is faithful.
\end{enumerate}
\end{prop}

\begin{proof}
$(1)\Leftrightarrow(2)$ is Proposition \ref{prop:chiral-mfc}.
 
$(1)\Rightarrow(3)$ For any closed $\Fun(\CD,\CD)$-$\Fun(\CC,\CC)$-bimodule $\CM$, $\CD^\op\boxtimes_{\Fun(\CD,\CD)}\CM\boxtimes_{\Fun(\CC,\CC)}\CC$ is a closed $n\Vect$-$n\Vect$-bimodule hence equivalent to $n\Vect$ by assumption. Therefore, $\CM \simeq \CD\boxtimes n\Vect\boxtimes\CC^\op \simeq \Fun(\CC,\CD)$.

$(3)\Rightarrow(1)$ An invertible separable $n$-category defines a closed $n\Vect$-$n\Vect$-bimodule which has to be trivial by assumption.

$(1)\Rightarrow(4)$ For $\CC,\CD\in\Sep^\ind_{n+1}$ we have $\Fun_{\Fun(\CD,\CD)|\Fun(\CC,\CC)}(\Fun(\CC,\CD),\Fun(\CC,\CD)) \simeq (n+1)\Vect$. Thus there is no nontrivial bimodule auto-equivalence of $\Fun(\CC,\CD)$.

$(4)\Rightarrow(1)$ An invertible separable $n$-category defines an $(n+1)\Vect$-$(n+1)\Vect$-bimodule auto-equivalence of $(n+1)\Vect$ which has to be trivial.
\end{proof}

\begin{exam} \label{exam:e0-fully-faithful}
(1) For $n=0$, $\FZ_0: \Sep^\ind_n \to \Fus^\cl_n$ is invertible. In fact, there is a unique indecomposable separable 0-category $\C$ and a unique nondegenerate fusion 0-category $\C$.

(2) For $n=1$, $\FZ_0: \Sep^\ind_n \to \Fus^\cl_n$ is invertible. In fact, there is a unique indecomposable separable 1-category $\Vect$ and a unique nondegenerate fusion 1-category $\Vect$.

(3) Condition (1) of Proposition \ref{prop:inv-z0} holds for $n=0,1,2$ but fails at least for $n=3,4$. Therefore, $\FZ_0$ is fully faithful but not essentially surjective for $n=2$; $\FZ_0$ is faithful but neither full nor essentially surjective for $n=3$; $\FZ_0$ is neither faithful nor full for $n=4$.
\end{exam}

\begin{rem}
The results of this subsection except Example \ref{exam:e0-fully-faithful} apply to an arbitrary base field $k$. 
However, when $k$ is not separably closed, the tensor products of $\Sep^\ind_n$ and $\Fus^\cl_n$ are ill-defined therefore $\FZ_0: \Sep^\ind_n \to \Fus^\cl_n$ is merely a functor.

The proof of $(4)\Rightarrow(3)$ of Proposition \ref{prop:fus-nondeg} only applies to a separably closed base field, but the general case can be reduced to the separably closed case by applying a field extension: the equation $\Sigma\FZ_1(\CA) = \FZ_0(\Sigma\CA)$ holds provided it holds after a field extension.
\end{rem}

\subsection{$E_1$-center functors}

Recall that given a braided monoidal $n$-category $\CA$, we use $\bar\CA$ to denote $\CA^{\op(-1)}$. Informally speaking, $\bar\CA$ is obtained from $\CA$ by reversing the braiding.

\begin{defn}
Let $\CA,\CB$ be braided fusion $n$-categories. A {\em multi-fusion $\CA$-$\CB$-bimodule} is a multi-fusion $n$-category $\CX$ equipped with a $\C$-linear braided monoidal functor $\psi_\CX:\CA\boxtimes\bar\CB\to\FZ_1(\CX)$. In the special case where $\psi_\CX$ is invertible, we say that $\CX$ is {\em closed}. We say that $\CA$ is {\em nondegenerate} if the multi-fusion $\CA$-$\CA$-bimodule $\CA$ is closed.
\end{defn}

\begin{rem}
A multi-fusion $\CA$-$\CB$-bimodule was called a multi-fusion $\CB$-$\CA$-bimodule in \cite{KZ18}. We change the convention to avoid unnecessary confusion of notations.
\end{rem}

\begin{prop} \label{prop:nondeg-bfc}
The following conditions are equivalent for a braided fusion $n$-category $\CA$:
\begin{enumerate}
\item $\CA$ is nondegenerate.
\item $\Sigma\CA$ is nondegenerate.
\item $\Sigma^2\CA$ is invertible.
\item $\FZ_0(\Sigma^2\CA)$ is trivial.
\item $\FZ_1(\Sigma\CA)$ is trivial.
\end{enumerate}
\end{prop}

\begin{proof}
$(1)\Leftrightarrow(2)$ By \cite[Lemma 3.50]{KZ22}, $\Sigma\FZ_1(\CA) = \FZ_0(\Sigma\CA)$. Therefore, the canonical braided monoidal functor $\CA\boxtimes\bar\CA \to \FZ_1(\CA)$ is invertible if and only if $\Sigma\CA\boxtimes\Sigma\bar\CA \to \FZ_0(\Sigma\CA)$ is invertible.
$(2)\Leftrightarrow(3)\Leftrightarrow(4)\Leftrightarrow(5)$ is due to Proposition \ref{prop:fus-nondeg}.
\end{proof}

\begin{cor} \label{cor:z1-nondeg}
If $\CA$ is an indecomposable multi-fusion $n$-category, then the braided fusion $n$-category $\FZ_1(\CA)$ is nondegenerate.
\end{cor}
\begin{proof}
By \cite[Lemma 3.50]{KZ22}, $\Sigma\FZ_1(\CA) = \FZ_0(\Sigma\CA)$, which is nondegenerate by Example \ref{exam:closed-bimod}.
\end{proof}

\begin{rem}
According to \cite[Corollary IV.3]{JF22}, Condition $(5)$ of Proposition \ref{prop:nondeg-bfc} is equivalent to the condition that $\FZ_2(\CA)$ is trivial. We give a different proof of this fact in Proposition \ref{prop:nondeg-bfc2}.
\end{rem}

\begin{defn}
We say that a nondegenerate braided fusion $n$-category is {\em chiral} if it is not the $E_1$-center of a fusion $n$-category. We say that two nondegenerate braided fusion $n$-categories $\CA$ and $\CB$ are {\em Witt equivalent} if $\CA\boxtimes\bar\CB$ is non-chiral.
\end{defn}

\begin{prop} \label{prop:chiral-bfc}
The following conditions are equivalent for a nondegenerate braided fusion $n$-category $\CA$:
\begin{enumerate}
\item $\CA$ is chiral.
\item $\Sigma\CA$ is chiral.
\item $\Sigma^2\CA$ is a nontrivial invertible separable $(n+2)$-category.
\end{enumerate}
\end{prop}
\begin{proof}
$(1)\Rightarrow(2)$ If $\Sigma\CA \simeq \FZ_0(\Sigma\CB)$ for a fusion $n$-category $\CB$, then $\CA\simeq\FZ_1(\CB)$ by \cite[Theorem 3.43]{KZ22}.

$(2)\Rightarrow(1)$ If $\CA\simeq\FZ_1(\CB)$ for a fusion $n$-category $\CB$, then $\Sigma\CA \simeq \Sigma\FZ_1(\CB) = \FZ_0(\Sigma\CB)$.

$(2)\Leftrightarrow(3)$ is due to Proposition \ref{prop:chiral-mfc}.
\end{proof}

\begin{prop} \label{prop:close-bimod2}
Let $\CA,\CB$ be nondegenerate braided fusion $n$-categories. A multi-fusion $\CA$-$\CB$-bimodule $\CM$ is closed if and only if the induced $\C$-linear functor $\Sigma\CM\boxtimes_{\Sigma\CA}-:\Sigma^2\CA\to\Sigma^2\CB$ is invertible.
\end{prop}
\begin{proof}
The multi-fusion $\CA$-$\CB$-bimodule $\CM$ is closed if and only if the $\Sigma\CA$-$\Sigma\CB$-bimodule $\Sigma\CM$ is closed, if and only if $\Sigma\CM\boxtimes_{\Sigma\CA}-:\Sigma^2\CA\to\Sigma^2\CB$ is invertible by Proposition \ref{prop:close-bimod1}. 
\end{proof}

Let $\MFus^\ind_n$ denote the symmetric monoidal 1-category where an object is an indecomposable multi-fusion $n$-category and a morphism is an equivalence class of nonzero separable bimodules. The tensor product is given by $\boxtimes$. The composition law is given by relative tensor product of bimodules.

By \cite[Theorem 3.22]{KZ22}, we have a symmetric monoidal equivalence:
$$\Sigma: \MFus^\ind_n \xrightarrow\sim \Sep^\ind_{n+1}, \quad \CA\mapsto\Sigma\CA\simeq\RMod_\CA((n+1)\Vect), \quad {}_\CA\CM_\CB\mapsto-\otimes_\CA\CM.$$

Let $\BFus^\cl_n$ denote the symmetric monoidal 1-category where an object is a nondegenerate braided fusion $n$-category $\CA$ and a morphism $\CA\to\CB$ is an equivalence class of closed multi-fusion $\CB$-$\CA$-bimodules. The tensor product is given by $\boxtimes$. The composition law is given by relative tensor product of bimodules, which is well-defined due to Proposition \ref{prop:close-bimod2}.

\begin{prop}
We have a symmetric monoidal equivalence
\begin{equation} \label{eqn:sigma-cl}
\Sigma: \BFus^\cl_n \xrightarrow\sim \Fus^\cl_{n+1}, \quad \CA\mapsto\Sigma\CA, \quad \CX\mapsto(\Sigma\CX,\CX).
\end{equation}
\end{prop}

\begin{proof}
The assignment \eqref{eqn:sigma-cl} is well-defined on objects by Proposition \ref{prop:nondeg-bfc}. 
For $\CC\in\Fus^\cl_{n+1}$, we have $\Fun(\CC,\CC)\simeq\CC\boxtimes\CC^\rev$, which is of fusion type by \cite[Lemma 3.49(1)]{KZ22}. Thus $\CC$ is connected so that $\CC \simeq \Sigma\Omega\CC$. By Proposition \ref{prop:nondeg-bfc}, \eqref{eqn:sigma-cl} is essentially surjective.

Let $\CX:\CA\to\CB$ be a morphism of $\BFus^\cl_n$. We have $(\Sigma\CA)^\rev\boxtimes\Sigma\CB \simeq \Sigma\FZ_1(\CX) \simeq \FZ_0(\Sigma\CX)$.
Therefore, $(\Sigma\CX,\CX):\Sigma\CA\to\Sigma\CB$ is a well-defined morphism of $\Fus^\cl_{n+1}$.

Conversely, let $(\CM,x):\Sigma\CA\to\Sigma\CB$ be a morphism of $\Fus^\cl_{n+1}$. Since $\Sigma\CA$ and $\Sigma\CB$ are indecomposable as separable $(n+1)$-categories, so is $\CM$. Thus $\FZ_1(\Omega(\CM,x)) \simeq \Omega\FZ_0(\Sigma\CM)\simeq\bar\CA\boxtimes\CB$, i.e. $\Omega(\CM,x):\CA\to\CB$ is a well-defined morphism of $\BFus^\cl_n$.

Since $\Omega(\Sigma\CX,\CX)\simeq\CX$ and $\Sigma\Omega(\CM,x)\simeq\CM$, the construction $\CX\mapsto(\Sigma\CX,\CX)$ on morphism is bijective. This shows that \eqref{eqn:sigma-cl} is fully faithful.

Let $\CX:\CA\to\CB$ and $\CY:\CB\to\CC$ be two morphisms of $\BFus^\cl_n$. We have $\Sigma(\CY\boxtimes_\CB\CX) \simeq \Sigma\CY\boxtimes_{\Sigma\CB}\Sigma\CX$ because $\Sigma$ is a left adjoint functor. Therefore, \eqref{eqn:sigma-cl} is a well-defined functor.
\end{proof}

\begin{cor} \label{cor:witt=morita=inv}
The construction $\CA\mapsto\Sigma\CA\mapsto\Sigma^2\CA$ gives a one-to-one correspondence between the following three types of data:
\begin{enumerate}
\renewcommand{\labelenumi}{$(\alph{enumi})$}
\item Witt equivalence classes of nondegenerate braided fusion $n$-categories.
\item Morita equivalence classes of nondegenerate fusion $(n+1)$-categories.
\item Equivalence classes of invertible separable $(n+2)$-categories.
\end{enumerate}
\end{cor}

\begin{thm} \label{thm:z1}
The assignment $\CA\mapsto\FZ_1(\CA)$ and ${}_\CA\CM_\CB\mapsto\Fun_{\CA|\CB}(\CM,\CM)$ defines a symmetric monoidal functor
$$\FZ_1: \MFus^\ind_n \to \BFus^\cl_n.$$
\end{thm}
\begin{proof}
According to Theorem \ref{thm:z0}, we have a composite symmetric monoidal functor
$$\MFus^\ind_n \simeq \Sep^\ind_{n+1} \xrightarrow{\FZ_0} \Fus^\cl_{n+1} \simeq \BFus^\cl_n.$$
The composition maps on objects
$$\CA \mapsto \Sigma\CA \mapsto \FZ_0(\Sigma\CA) \mapsto \Omega\FZ_0(\Sigma\CA) = \FZ_1(\CA),$$
and maps on morphisms
$${}_\CA\CM_\CB \mapsto (\Sigma\CA\xrightarrow{-\boxtimes_\CA\CM}\Sigma\CB) \mapsto (\BMod_{\CA|\CB}((n+1)\Vect),\CM) $$ 
$$\simeq \Sigma\Fun_{\CA|\CB}(\CM,\CM) \mapsto \Fun_{\CA|\CB}(\CM,\CM),$$
as desired.
\end{proof}

\begin{rem}
Unwinding the constructions, we see that an isomorphism $\CM:\CA\to\CB$ in $\MFus^\ind_n$, i.e. an invertible bimodule, induces a $\C$-linear equivalence $\Sigma\CA\simeq\Sigma\CB$ then a monoidal equivalence $\FZ_0(\Sigma\CA)\simeq\FZ_0(\Sigma\CB)$ and then a braided monoidal equivalence $\FZ_1(\CA)\simeq\FZ_1(\CB)$ giving rise to the isomorphism $\FZ_1(\CM):\FZ_1(\CA)\to\FZ_1(\CB)$ in $\BFus^\cl_n$.
\end{rem}


\begin{cor} \label{cor:e1-center}
Let $\CA,\CB,\CC$ be multi-fusion $n$-categories where $\CB$ is indecomposable. We have an equivalence of multi-fusion $\FZ_1(\CC)$-$\FZ_1(\CA)$-bimodules 
$$\Fun_{\CB|\CC}(\CN,\CN) \boxtimes_{\FZ_1(\CB)} \Fun_{\CA|\CB}(\CM,\CM) \simeq \Fun_{\CA|\CC}(\CM\boxtimes_\CB\CN,\CM\boxtimes_\CB\CN), \quad g\boxtimes_{\FZ_1(\CB)}f \mapsto f\boxtimes_\CB g$$
for separable bimodules ${}_\CA\CM_\CB$ and ${}_\CB\CN_\CC$.
Moreover, we have a $\C$-linear equivalence
$$\Fun_{\CB|\CC}(\CQ,\CQ') \boxtimes_{\FZ_1(\CB)} \Fun_{\CA|\CB}(\CP,\CP') \simeq \Fun_{\CA|\CC}(\CP\boxtimes_\CB\CQ,\CP'\boxtimes_\CB\CQ'), \quad g\boxtimes_{\FZ_1(\CB)}f \mapsto f\boxtimes_\CB g$$
for separable bimodules ${}_\CA\CP_\CB$, ${}_\CA\CP'_\CB$, ${}_\CB\CQ_\CC$, ${}_\CB\CQ'_\CC$.
\end{cor}
\begin{proof}
We may assume that $\CA,\CC$ are indecomposable and that $\CM,\CN$ are nonzero. Then the first claim follows from Theorem \ref{thm:z1}. Taking $\CM=\CP\oplus\CP'$ and $\CN=\CQ\oplus\CQ'$ we obtain the second claim.
\end{proof}

\begin{cor} \label{cor:fun-centralizer}
Let $F:\CA\to\CB$ be a $\C$-linear monoidal functor between multi-fusion $n$-categories where $\CA$ is indecomposable. We have a $\C$-linear monoidal equivalence
$$\Fun_{\CA|\CB}(\CB,\CB) \boxtimes_{\FZ_1(\CA)} \CA \simeq \CB, \quad f\boxtimes_{\FZ_1(\CA)}X \mapsto f F(X).$$
\end{cor}
\begin{proof}
Use $\CA \simeq \Fun_{n\Vect|\CA}(\CA,\CA)$ and $\CB \simeq \Fun_{n\Vect|\CB}(\CA\boxtimes_\CA\CB,\CA\boxtimes_\CA\CB)$.
\end{proof}

\begin{rem} \label{rem:fun-centralizer}
In Corollary \ref{cor:fun-centralizer}, the composition $\CA \xrightarrow{\Id_\CB\boxtimes_{\FZ_1(\CA)}\Id_\CA} \Fun_{\CA|\CB}(\CB,\CB)\boxtimes_{\FZ_1(\CA)}\CA \simeq \CB$ is equivalent to $F$.
Conversely, if there is a multi-fusion $n\Vect$-$\FZ_1(\CA)$-bimodule $\CE$ and a $\C$-linear monoidal equivalence $\CE\boxtimes_{\FZ_1(\CA)}\CA \simeq \CB$ such that the composition $\CA \xrightarrow{\one_\CE\boxtimes_{\FZ_1(\CA)}\Id_\CA} \CE\boxtimes_{\FZ_1(\CA)}\CA \simeq \CB$ is equivalent to $F$, then $\CE \simeq \CE\boxtimes_{\FZ_1(\CA)}\Fun_{\CA|\CA}(\CA,\CA) \simeq \Fun_{\CA|\CA}(\CA,\CE\boxtimes_{\FZ_1(\CA)}\CA) \simeq \Fun_{\CA|\CA}(\CA,\CB) \simeq \Fun_{\CA|\CB}(\CB,\CB)$.
\end{rem}

\begin{rem}
This subsection relies heavily on \cite[Lemma 3.49 and Lemma 3.50]{KZ22}, thus the results only apply to separably closed base fields. 
However, Proposition \ref{prop:nondeg-bfc} and Proposition \ref{prop:chiral-bfc} apply to an arbitrary base field, because nondegeneracy is stable under field extensions and because the equation $\Sigma\FZ_1(\CA) = \FZ_0(\Sigma\CA)$ holds provided it holds after a field extension.
Proposition \ref{prop:close-bimod2} also works for an arbitrary base field.
We conjecture that Corollary \ref{cor:witt=morita=inv} holds for an arbitrary base field.
\end{rem}

\begin{rem}
Direction $(2)\Leftrightarrow(3)$ of Proposition \ref{prop:fus-nondeg} is due to \cite[Theorem 2]{JF22}.
Direction $(1)\Leftrightarrow(4)$ of Proposition \ref{prop:nondeg-bfc} for $n=1$ is due to \cite[Proposition 4.16]{DN20}.
Corollary \ref{cor:z1-nondeg} for $n=1$ is due to \cite[Proposition 4.4]{ENO04}.
The notion of Witt equivalence was introduced in \cite{DMNO13} for nondegenerate braided fusion 1-categories.
Theorem \ref{thm:z1} for $n=1$ is a main result of \cite{KZ18}.
\end{rem}

\section{Centralizers and modular extensions} \label{sec:cen-mext}

The notion of a modular extension involves a pair of braided fusion $n$-categories centralizing each other within a nondegenerate one. Driven by this challenge, we develop the mathematical theory of $E_m$-centralizers in the first subsection which is largely parallel to that of $E_m$-centers developed in \cite{KZ22}. Then we introduce and study the notion of perfect pairing in the next subsection. With these preparations, we explore the theory of modular extensions in the final subsection. In particular, we achieve several classification outcomes, encompassing those conjectures posited in the physics literature. Once more, we employ the tactic of simplifying the original problem to the relatively straightforward $E_0$-context.

\subsection{Centralizers} \label{sec:centralizers}

The following definition is standard. See \cite[Section 5.3]{Lur14}. 

\begin{defn} \label{defn:centralizers}
Let $F:\CA\to\CB$ be a $\C$-linear $E_m$-monoidal functor where $\CA,\CB \in E_m\KarCat^\C_n$. The {\em $E_m$-centralizer} of $F$ is the universal condensation-complete $\C$-linear $E_m$-monoidal $n$-category $\FZ_m(F)$ equipped with a unital action $G:\FZ_m(F)\boxtimes\CA\to\CB$, i.e. a $\C$-linear $E_m$-monoidal functor rendering the following diagram in $E_m\KarCat^\C_n$ commutative up to equivalence:
$$\xymatrix{
  & \FZ_m(F)\boxtimes\CA \ar[rd]^G \\
  \CA \ar[ru]^{\one_{\FZ_m(F)}\boxtimes\Id_\CA} \ar[rr]^{F} && \CB .
}
$$
By slightly abusing notation, we use $\FZ_m(\CA,\CB)$ to denote $\FZ_m(F)$ and refer to it as the $E_m$-centralizer of $\CA$ in $\CB$ when $F$ is clear from the context.
\end{defn}

\begin{rem}
The $E_m$-centralizer $\FZ_m(\Id_\CA)$ is the $E_m$-center $\FZ_m(\CA)$ for $\CA \in E_m\KarCat^\C_n$.
The $E_m$-centralizer of the canonical $\C$-linear $E_m$-monoidal functor $n\Vect\to\CA$ is $\CA$.
\end{rem}

\begin{rem}
By definition, $\FZ_m(\CA^{\op k},\CB^{\op k}) = \FZ_m(\CA,\CB)^{\op k}$. 
\end{rem}

\begin{exam}
For $m=0$, we have $\FZ_0(F) = (\Fun(\CA,\CB),F)$. In fact, giving a unital action $\CC\boxtimes\CA\to\CB$ is equivalent to giving a $\C$-linear functor $\CC\to\Fun(\CA,\CB)$ that maps $\one_\CC$ to $F$.
\end{exam}

\begin{exam} \label{exam:centralizer-e1}
For $\CA\in E_1\KarCat^\C_n$ and $\CM\in\LMod_\CA(\KarCat^\C_n)$, 
we have $\FZ_1(\CA,\FZ_0(\CM)) = \Fun_\CA(\CM,\CM)$. In fact, giving a unital action $\CC\boxtimes\CA\to\FZ_0(\CM)$ is equivalent to giving a $\C$-linear monoidal functor $\CC\to\Fun_\CA(\CM,\CM)$.
\end{exam}

\begin{exam} \label{exam:centralizer-e2}
For $n=1$, $\FZ_2(\CA,\CB)$ is the M\"uger centralizer of $\CA$ in $\CB$, i.e. the maximal full subcategory of $\CB$ that is transparent to the essential image of $F:\CA\to\CB$. This is because, for a unital action $\CC\boxtimes\CA\to\CB$, the essential images of $\CC$ and $\CA$ in $\CB$ are transparent to each other as $\CC$ and $\CA$ are transparent to each other in $\CC\boxtimes\CA$. However, $\FZ_2(\CA,\CB)$ is no longer a full subcategory of $\CB$ for $n\ge2$ since the transparency is data rather than a property.
\end{exam}

\begin{exam}
For $m=\infty$, we have $\FZ_\infty(\CA,\CB) = \CB$. In particular, $\FZ_\infty(\CA) = \CA$. In fact, giving a $\C$-linear symmetric monoidal functor $\CC\boxtimes\CA\to\CB$ is equivalent to giving a pair of $\C$-linear symmetric monoidal functors $\CC\to\CB$ and $\CA\to\CB$. In another word, $\CC\boxtimes\CA$ is the coproduct of $\CC$ and $\CA$ in $E_\infty\KarCat^\C_n$ (see \cite[Proposition 3.2.4.7]{Lur14}).
\end{exam}

\begin{lem} \label{lem:centralizer-loop}
Let $F:\CA\to\Omega\CB$ be a $\C$-linear $E_m$-monoidal functor where $\CA\in E_m\KarCat^\C_n$ and $\CB \in E_{m-1}\KarCat^\C_{n+1}$. We have $\FZ_m(\CA,\Omega\CB) = \Omega\FZ_{m-1}(\Sigma\CA,\CB)$.
\end{lem}
\begin{proof}
Since $\Sigma$ is left adjoint to $\Omega$, we have 
$$\Fun^{E_m}(\CC\boxtimes\CA,\Omega\CB) \simeq \Fun^{E_{m-1}}(\Sigma\CC\boxtimes\Sigma\CA,\CB)$$ 
for $\CC\in E_m\KarCat^\C_n$.
Therefore, 
$$\Fun^{E_m}(\CC,\FZ_m(\CA,\Omega\CB)) \simeq \Fun^{E_{m-1}}(\Sigma\CC,\FZ_{m-1}(\Sigma\CA,\CB))$$
where the right hand side is equivalent to $\Fun^{E_m}(\CC,\Omega\FZ_{m-1}(\Sigma\CA,\CB))$.
Hence $\FZ_m(\CA,\Omega\CB) = \Omega\FZ_{m-1}(\Sigma\CA,\CB)$.
\end{proof}

\begin{thm} \label{thm:mfc-centralizer}
We have $\FZ_m(\CA,\CB) = \Omega^k\FZ_{m-k}(\Sigma^k\CA,\Sigma^k\CB)$ where $0\le k\le m$. In particular, $\FZ_m(\CA,\CB) = \Omega^m(\Fun(\Sigma^m\CA,\Sigma^m\CB),\Sigma^m F)$.
\end{thm}

\begin{proof}
Apply Lemma \ref{lem:centralizer-loop} for $k$ times.
\end{proof}

\begin{exam} \label{exam:centralizer:z1}
In the special case where $m=1$, we have a pullback diagram
$$\xymatrix{
  \Fun(\Sigma\CA,\Sigma\CB) \ar@{^(->}[r] \ar[d] & \BMod_{\CA|\CB}(\KarCat^\C_n) \ar[d] \\
  \Sigma\CB \ar@{^(->}[r] & \RMod_\CB(\KarCat^\C_n) .
}$$
That is, giving a $\C$-linear functor $\Sigma\CA\to\Sigma\CB$ is equivalent to giving an $\CA$-$\CB$-bimodule contained in $\Sigma\CB$. Therefore, $\FZ_1(\CA,\CB) = \Omega(\Fun(\Sigma\CA,\Sigma\CB),\Sigma F) \simeq \Fun_{\CA|\CB}(\CB,\CB) \simeq \Fun_{\CA|\CA}(\CA,\CB)$. 
\end{exam}

\begin{cor}
If $\CA$ and $\CB$ are $E_m$-multi-fusion $n$-categories, then $\FZ_m(\CA,\CB)$ is also an $E_m$-multi-fusion $n$-category. 
\end{cor}

\begin{cor}
Let $F:\CA\boxtimes\CB\to\CC$ be a $\C$-linear $E_m$-monoidal functor where $\CA,\CB,\CC \in E_m\KarCat^\C_n$. Then $\FZ_m(\CA\boxtimes\CB,\CC) = \FZ_m(\CA,\FZ_m(\CB,\CC))$.
\end{cor}
\begin{proof}
Combine Theorem \ref{thm:mfc-centralizer} and the equivalence $\Fun(\Sigma^m\CA\boxtimes\Sigma^m\CB,\Sigma^m\CC) = \Fun(\Sigma^m\CA,\Fun(\Sigma^m\CB,\Sigma^m\CC))$.
\end{proof}

\begin{prop} \label{prop:centralizer-center}
Let $\CM$ be a condensation-complete $\C$-linear $E_m$-monoidal $n$-category over $\CA^{\op(-m)}$ where $\CA\in E_{m+1}\KarCat^\C_n$. We have 
$$\FZ_{m+1}(\CA^{\op(-m)},\FZ_m(\CM)) = \Omega\FZ_m(\Sigma\CA^{\op(-m)},\FZ_{m-1}(\Sigma\CM))$$
for $m\ge1$. Therefore, 
$$\FZ_{m+1}(\CA^{\op(-m)},\FZ_m(\CM)) 
= \Omega^{m+1}(\Sigma^{m+1}\CA,\Sigma^m\CM).$$
\end{prop}
\begin{proof}
Combining Lemma \ref{lem:centralizer-loop} with the equivalence $\FZ_m(\CM) = \Omega\FZ_{m-1}(\Sigma\CM)$, we obtain the first identity. The second identity follows from the first one and Example \ref{exam:centralizer-e1}.
\end{proof}

\begin{lem} \label{lem:e1-cent}
Let $F:\CA\to\CB$ be a $\C$-linear monoidal functor between indecomposable multi-fusion $n$-categories. Then $\Sigma\FZ_1(\CA,\CB)=\FZ_0(\Sigma\CA,\Sigma\CB)$.
\end{lem}
\begin{proof}
We have $\Fun(\Sigma\CA,\Sigma\CB) \simeq (\Sigma\CA)^\op\boxtimes\Sigma\CB$, which is indecomposable by \cite[Lemma 3.49(1)]{KZ22}. Hence $\Sigma\FZ_1(\CA,\CB) = \Sigma\Omega\FZ_0(\Sigma\CA,\Sigma\CB) = \FZ_0(\Sigma\CA,\Sigma\CB)$.
\end{proof}

\begin{prop}
Let $F:\CA\to\CB$ be a $\C$-linear monoidal functor between multi-fusion $n$-categories where $\CA$ is nondegenerate then $\CB \simeq \FZ_1(\CA,\CB)\boxtimes\CA$. In particular, $\CB$ is nondegenerate if and only if $\FZ_1(\CA,\CB)$ is nondegenerate.
\end{prop}
\begin{proof}
Since $\Sigma\CA$ is invertible, we have $\Sigma\CB \simeq \FZ_0(\Sigma\CA,\Sigma\CB)\boxtimes\Sigma\CA$ which is equivalent by Lemma \ref{lem:e1-cent} to $\Sigma\FZ_1(\CA,\CB)\boxtimes\Sigma\CA$. Therefore, $\CB \simeq \FZ_1(\CA,\CB)\boxtimes\CA$.
\end{proof}

\begin{cor} \label{cor:braid-nondeg}
Let $F:\CA\to\CB$ be a $\C$-linear braided monoidal functor between braided fusion $n$-categories where $\CA$ is nondegenerate. Then $\CB \simeq \FZ_2(\CA,\CB)\boxtimes\CA$. In particular, $\CB$ is nondegenerate if and only if $\FZ_2(\CA,\CB)$ is nondegenerate.
\end{cor}

\begin{rem}
When the base field is not separably closed, Lemma \ref{lem:e1-cent} fails as \cite[Lemma 3.49(1)]{KZ22} fails. However, when either $\CA$ or $\CB$ is nondegenerate, \cite[Lemma 3.49(1)]{KZ22} can be bypassed hence Lemma \ref{lem:e1-cent} is always true.
\end{rem}

\begin{rem}
Corollary \ref{cor:braid-nondeg} for $n=1$ is due to \cite[Proposition 4.1]{Mu03b}, \cite[Corollary 3.26]{DMNO13} and \cite[Theorem 3.13]{DGNO10}.
\end{rem}

\subsection{Perfect pairing} \label{sec:pairing}

\begin{defn} \label{defn:pairing}
Let $\CA,\CB,\CM \in E_m\KarCat^\C_n$. We say that a $\C$-linear $E_m$-monoidal functor $\phi:\CA\boxtimes\CB\to\CM$ is a {\em perfect pairing} if the induced $\C$-linear $E_m$-monoidal functors $\CA\to\FZ_m(\CB,\CM)$ and $\CB\to\FZ_m(\CA,\CM)$ are both invertible, i.e. $\phi$ exhibits $\CA$ and $\CB$ as $E_m$-centralizers of each other in $\CM$.
\end{defn}

\begin{exam}
According to Proposition \ref{prop:centralizer-center}, the canonical $\C$-linear $E_{m+1}$-monoidal functor $\CA\boxtimes\CA^{\op(-m)}\to\FZ_m(\CA)$ is a perfect pairing for any $\CA\in E_{m+1}\KarCat^\C_n$.
\end{exam}

\begin{prop} \label{prop:e1-centralizer}
Let $\CA,\CB,\CM$ be indecomposable multi-fusion $n$-categories where $\CM$ is nondegenerate. A $\C$-linear monoidal functor $\phi:\CA\boxtimes\CB\to\CM$ is a perfect pairing if and only if $\Sigma\phi:\Sigma\CA\boxtimes\Sigma\CB\to\Sigma\CM$ is a perfect pairing.
\end{prop}
\begin{proof}
By Lemma \ref{lem:e1-cent}, $\Sigma\FZ_1(\CB,\CM) = \FZ_0(\Sigma\CB,\Sigma\CM)$. Then an equivalence $\Sigma\CA\simeq\FZ_0(\Sigma\CB,\Sigma\CM)$ induces an equivalence $\CA\simeq\FZ_1(\CB,\CM)$ and vice versa.
\end{proof}

\begin{defn}
We say that an $E_0$-multi-fusion $n$-category $\CA$ is {\em nondegenerate} if the underlying separable $n$-category of $\CA$ is invertible. 
\end{defn}

\begin{prop} \label{prop:e0-per-pair}
Let $F:\CA\to\CB$ be a $\C$-linear $E_0$-monoidal functor between $E_0$-multi-fusion $n$-categories where $\CB$ is nondegenerate. Then the canonical $\C$-linear $E_0$-monoidal functor $\FZ_0(\CA,\CB)\boxtimes\CA\to\CB$ is a perfect pairing.
\end{prop}
\begin{proof}
Since $\CB$ is an invertible separable $n$-category, the canonical $\C$-linear functor $\CA \to \Fun(\Fun(\CA,\CB),\CB)$ is invertible. Hence the canonical $\C$-linear $E_0$-monoidal functor $\CA \to \FZ_0(\FZ_0(\CA,\CB),\CB)$ is invertible.
\end{proof}

\begin{cor} \label{cor:e1-per-pair}
Let $F:\CA\to\CB$ be a $\C$-linear monoidal functor between indecomposable multi-fusion $n$-categories where $\CB$ is nondegenerate. Then the canonical $\C$-linear monoidal functor $\FZ_1(\CA,\CB)\boxtimes\CA\to\CB$ is a perfect pairing.
\end{cor}
\begin{proof}
Combine Lemma \ref{lem:e1-cent}, Proposition \ref{prop:e1-centralizer} and Proposition \ref{prop:e0-per-pair}.
\end{proof}

\begin{lem} \label{lem:centralizer-ff}
Let $F:\CA\to\CB$ and $G:\CB\to\CC$ be $\C$-linear $E_m$-monoidal functors where $\CA,\CB,\CC \in E_m\KarCat^\C_n$. If $G$ is fully faithful, then the induced $\C$-linear $E_m$-monoidal functor $G':\FZ_m(\CA,\CB)\to\FZ_m(\CA,\CC)$ is also fully faithful.
\end{lem}
\begin{proof}
Let $\CE \subset \FZ_m(\CA,\CC)$ be the full subcategory consisting of the condensates of the image of $G'$. Then the unital action $\FZ_m(\CA,\CC)\boxtimes\CA\to\CC$ restricts to a unital action $\CE\boxtimes\CA\to\CB$ so that $\CE$ shares the same universal property as $\FZ_m(\CA,\CB)$. Therefore, $G'$ is fully faithful onto $\CE$.
\end{proof}

\begin{prop} \label{prop:e2-centralizer}
Let $\CA,\CB,\CM$ be braided fusion $n$-categories where $\CM$ is nondegenerate. A $\C$-linear braided monoidal functor $\phi:\CA\boxtimes\CB\to\CM$ is a perfect pairing if and only if $\Sigma\phi:\Sigma\CA\boxtimes\Sigma\CB\to\Sigma\CM$ is a perfect pairing.
\end{prop}
\begin{proof}
Sufficiency. According to theorem \ref{thm:mfc-centralizer}, an equivalence $\Sigma\CA\simeq\FZ_1(\Sigma\CB,\Sigma\CM)$ induces an equivalence $\CA\simeq\FZ_2(\CB,\CM)$. Therefore, if $\Sigma\phi$ is a perfect pairing, so is $\phi$.

Necessity. Let $\CA'=\FZ_1(\Sigma\CB,\Sigma\CM)$ and $\CB'=\FZ_1(\Sigma\CA,\Sigma\CM)$ so that $\CA\simeq\Omega\CA'$ and $\CB\simeq\Omega\CB'$. We need to show that $\CA'$ is connected whence the induced $\C$-linear monoidal functor $\Sigma\CA\hookrightarrow\CA'$ is invertible.
By Corollary \ref{cor:e1-per-pair}, $\FZ_1(\CA',\Sigma\CM)\simeq\Sigma\CB$ and $\FZ_1(\CB',\Sigma\CM)\simeq\Sigma\CA$. We have a commutative diagram in $\MFus^\ind_n$:
$$\xymatrixcolsep{5em}\xymatrix{
  \Sigma\CA \ar[r]^-{\CA'} \ar[d]_{\Sigma\CM}^\sim & \CA' \ar[d]^{\Sigma\CM}_\sim \\
  \CB'^\rev\boxtimes\Sigma\CM \ar[r]^-{\CB'\boxtimes\Sigma\CM} & \Sigma\bar\CB\boxtimes\Sigma\CM
}
$$
where the vertical morphisms are invertible by Proposition \ref{prop:fun-inv}(1). 
By the functoriality of $E_1$-center, we obtain a commutative diagram (up to equivalence) of fusion $(n+1)$-categories
$$\xymatrix{
  \FZ_1(\CA') \ar[r]^-\psi \ar[d]_\sim & \FZ_1(\Sigma\CA,\CA') \ar[d]^\sim \\
  \overline{\FZ_1(\Sigma\CB\boxtimes\Sigma\bar\CM)} \ar[r]^-{\psi'} & \FZ_1(\Sigma\CB\boxtimes\Sigma\bar\CM,\CB'\boxtimes\Sigma\bar\CM) .  
}
$$
Since $\psi'$ is fully faithful by Lemma \ref{lem:centralizer-ff}, $\psi$ is also fully faithful. Thus we have a fully faithful $\C$-linear monoidal functor $\CA'^\rev \hookrightarrow \CA'^\rev\boxtimes_{\FZ_1(\CA')}\FZ_1(\Sigma\CA,\CA') \simeq \Fun_{\CA'}(\CA',\CA')\boxtimes_{\FZ_1(\CA')}\Fun_{\Sigma\CA|\CA'}(\CA',\CA') \simeq \Fun_{\Sigma\CA}(\CA',\CA')$, where we applied the functoriality of $E_1$-center again. In particular, $\Fun_{\Sigma\CA}(\CA',\CA')$ has a simple tensor unit, i.e. $\CA'$ is connected, as desired.
\end{proof}

\begin{rem}
The proof of Proposition \ref{prop:e2-centralizer} only applies to a separably closed base field, but the general case can be reduced to the separably closed case by applying a field extension because field extensions are compatible with centralizers. 
\end{rem}

\subsection{Modular extensions} \label{sec:mext}

\begin{defn}[$E_0$-version] \label{defn:e0-mext} 
Let $\CA$ be a nonzero $E_0$-multi-fusion $n$-category over $\CE^\rev$ defined by a $\C$-linear monoidal functor $\psi_\CA:\CE^\rev\to\FZ_0(\CA)$, where $\CE$ is a nonzero multi-fusion $n$-category. An {\em $E_0$-modular extension} of $\CA$ is a nondegenerate $E_0$-multi-fusion $n$-category $\CM$ equipped with a $\C$-linear $E_0$-monoidal functor $\CA\to\CM$ such that the composite $\C$-linear $E_0$-monoidal functor $\CA\boxtimes\CE\to\CA\to\CM$ is a perfect pairing. Two modular extensions $\CM$ and $\CN$ of $\CA$ are {\em equivalent} if there is a $\C$-linear $E_0$-monoidal equivalence $\CM\simeq\CN$ such that the composition $\CA\to\CM\simeq\CN$ is equivalent to $\CA\to\CN$.
\end{defn}

\begin{exam} \label{exam:e0-mext}
In the special case where $\CA=(\CE^\op,X)$ equipped with the canonical monoidal functor $\psi_\CA:\CE^\rev\to\FZ_0(\CA)$, $(n\Vect,\Hom_\CE(X,\one_\CE))$ is an $E_0$-modular extension of $\CA$ because the evaluation functor $\CA\boxtimes\CE\to(n\Vect,\Hom_\CE(X,\one_\CE))$ is a perfect pairing.
\end{exam}

\begin{thm} \label{thm:mod-ext-e0}
The equivalence classes of $E_0$-modular extensions of $\CA$ one-to-one correspond to the equivalence classes of $\C$-linear monoidal equivalences $\FZ_0(\CE)^\rev \simeq \FZ_0(\CA)$ such that the composition $\CE^\rev \to \FZ_0(\CE)^\rev \simeq \FZ_0(\CA)$ is equivalent to $\psi_\CA$.
\end{thm}

\begin{proof}
Given such an equivalence $\FZ_0(\CE)^\rev \simeq \FZ_0(\CA)$, we let $\CM = \CA\boxtimes_{\FZ_0(\CE)}\CE$ equipped with the $\C$-linear $E_0$-monoidal functor $\Id_\CA\boxtimes_{\FZ_0(\CE)}\one_\CE:\CA\to\CM$. Then $\CM$ is a closed $n\Vect$-$n\Vect$-bimodule by Corollary \ref{cor:closed-bimod}, hence $\CM$ is nondegenerate by Corollary \ref{cor:inv-trivial}. Moreover, $\FZ_0(\CE,\CM) \simeq \CA\boxtimes_{\FZ_0(\CE)}\FZ_0(\CE,\CE) \simeq \CA$ by Proposition \ref{prop:fun-tensor} and similarly $\FZ_0(\CA,\CM)\simeq\CE$. That is, the canonical $\C$-linear $E_0$-monoidal functor $\CA\boxtimes\CE\to\CM$ is a perfect pairing. Since the composition $\CE^\rev \to \FZ_0(\CE)^\rev \simeq \FZ_0(\CA)$ is equivalent to $\psi_\CA$, we have a natural equivalence $A\boxtimes_{\FZ_0(\CE)}E \simeq (A\otimes E)\boxtimes_{\FZ_0(\CE)}\one_\CE$ for $A\in\CA$ and $E\in\CE$. Therefore, the composition $\CA\boxtimes\CE\to\CA\to\CM$ is equivalent to the canonical one hence is also a perfect pairing. This shows that $\CM$ is an $E_0$-modular extension of $\CA$.

We need to show the equivalence class of $\CM$ over $\CA$ depends only on the equivalence class of the given $\FZ_0(\CE)^\rev \simeq \FZ_0(\CA)$. Indeed, a $\C$-linear monoidal equivalence $\phi:\FZ_0(\CE)^\rev \to \FZ_0(\CA)$ promotes $\CA$ to a right $\FZ_0(\CE)$-module, say, $\CA_\phi$. A monoidal natural equivalence $\phi\simeq\varphi$ promotes the identity functor of $\CA$ to a right $\FZ_0(\CE)$-module equivalence $\CA_\phi\simeq\CA_\varphi$. Therefore, $\CA_\phi\boxtimes_{\FZ_0(\CE)}\CE \simeq \CA_\varphi\boxtimes_{\FZ_0(\CE)}\CE$ over $\CA$, as desired. 

Conversely, given an $E_0$-modular extension $\CM$ of $\CA$, we see that the composition $\CA\boxtimes\CE\to\CA\to\CM$ induces a left $\CE$-module equivalence $\CE\simeq\FZ_0(\CA,\CM)$. Then $\CE$ is a closed $\FZ_0(\CM)$-$\FZ_0(\CA)$-bimodule by Example \ref{exam:closed-bimod}. Since $\CM$ is nondegenerate, $\FZ_0(\CM)$ is trivial. We obtain $\FZ_0(\CA) \simeq \FZ_0(\CE)^\rev$. Since $\CE\simeq\FZ_0(\CA,\CM)$ as left $\CE$-modules, the composition $\psi:\CE^\rev \to \FZ_0(\CE)^\rev \simeq \FZ_0(\CA)$ is equivalent to $\psi_\CA$. Since $\FZ_0(\CA) \simeq \FZ_0(\CE)^\rev$ is determined by the right $\FZ_0(\CA)$-action on $\FZ_0(\CA,\CM)$, its equivalence class depends only on the equivalence class of $\CM$ over $\CA$.

It remains to demonstrate that the two constructions are mutually inverse. Indeed, given an $E_0$-modular extension $\CM$ of $\CA$, we have $\CA\boxtimes_{\FZ_0(\CE)}\CE \simeq \CA\boxtimes_{\FZ_0(\CA)^\rev}\FZ_0(\CA,\CM) \simeq \CM$ over $\CA$ where we applied the functoriality of $E_0$-center. On the other hand, for a $\C$-linear monoidal equivalence $\phi:\FZ_0(\CE)^\rev \to \FZ_0(\CA)$, the right action of $\FZ_0(\CA)$ on $\FZ_0(\CA,\CA\boxtimes_{\FZ_0(\CE)}\CE) \simeq \CE$ recovers $\phi$. This completes the proof.
\end{proof}

\begin{exam}
According to Corollary \ref{cor:inv-trivial}, if $\CA$ itself is nondegenerate, then $\CA$ admits an $E_0$-modular extension if and only if $\CE$ is trivial. Moreover, the only $E_0$-modular extension of $\CA$ is $\CA$ itself.
\end{exam}

\begin{rem} \label{rem:mext-e0-group}
Let $\MExt^{E_0}(\CE)$ be the group formed by the (equivalence classes of) $\C$-linear monoidal equivalences $\phi:\FZ_0(\CE)\to\FZ_0(\CE)$ such that the composition $\CE\to\FZ_0(\CE)\xrightarrow\phi\FZ_0(\CE)$ is equivalent to $\CE\to\FZ_0(\CE)$. Then Theorem \ref{thm:mod-ext-e0} implies that the collection of (equivalence classes of) $E_0$-modular extensions of $\CA$, if they exist, form a torsor over $\MExt^{E_0}(\CE)$. In particular, in the situation of Example \ref{exam:e0-mext} where $(\CE^\op,X)$ admits a canonical $E_0$-modular extension, the $E_0$-modular extensions of $(\CE^\op,X)$ form a group that is canonically identified with $\MExt^{E_0}(\CE)$.
\end{rem}

\begin{defn}[$E_1$-version] \label{defn:e1-mext}
Let $\CA$ be an indecomposable multi-fusion $n$-category over $\bar\CE$ where $\CE$ is a braided fusion $n$-category. An {\em $E_1$-modular extension} of $\CA$ is a nondegenerate multi-fusion $n$-category $\CM$ equipped with a $\C$-linear monoidal functor $\CA\to\CM$ such that the composite $\C$-linear monoidal functor $\CA\boxtimes\CE\to\CA\to\CM$ is a perfect pairing. Two $E_1$-modular extensions $\CM$ and $\CN$ of $\CA$ are {\em equivalent} if there is a $\C$-linear monoidal equivalence $\CM\simeq\CN$ such that the composition $\CA\to\CM\simeq\CN$ is equivalent to $\CA\to\CN$.
\end{defn}

\begin{exam}
In the special case where $\CA=\CE^\rev$ equipped with the canonical braided monoidal functor $\psi_\CA:\bar\CE\to\overline{\FZ_1(\CE)}\simeq\FZ_1(\CE^\rev)$, $\FZ_0(\CE)$ is an $E_1$-modular extension of $\CE^\rev$. In fact, according to Example \ref{exam:centralizer-e1}, $\CE$ and $\CE^\rev$ centralize each other in $\FZ_0(\CE)$.
\end{exam}

\begin{cor} \label{cor:mod-ext-e1}
The equivalence classes of $E_1$-modular extensions of $\CA$ one-to-one correspond to the equivalence classes of $\C$-linear braided monoidal equivalences $\overline{\FZ_1(\CE)} \simeq \FZ_1(\CA)$ such that the composition $\bar\CE \to \overline{\FZ_1(\CE)} \simeq \FZ_1(\CA)$ is equivalent to $\psi_\CA$.
\end{cor}
\begin{proof}
Note that $\Sigma\CA$ is an $E_0$-multi-fusion $(n+1)$-category over $(\Sigma\CE)^\rev$. According to Proposition \ref{prop:e1-centralizer}, giving an $E_1$-modular extension of $\CA$ is equivalent to giving an $E_0$-modular extension of $\Sigma\CA$. Moreover, giving a $\C$-linear braided monoidal equivalence $\overline{\FZ_1(\CE)} \simeq \FZ_1(\CA)$ is equivalent to giving a $\C$-linear monoidal equivalence $\FZ_0(\Sigma\CE)^\rev \simeq \FZ_0(\Sigma\CA)$. Therefore, the claim follows from Theorem \ref{thm:mod-ext-e0}.
\end{proof}

\begin{rem}
Unwinding the proof we see that the $E_1$-modular extension of $\CA$ corresponding to $\overline{\FZ_1(\CE)} \simeq \FZ_1(\CA)$ is $\CA\boxtimes_{\FZ_1(\CE)}\CE$.
\end{rem}

\begin{exam}
If $\CA$ itself is nondegenerate, then $\CA$ admits an $E_1$-modular extension if and only if $\CE$ is trivial. Moreover, the only $E_1$-modular extension of $\CA$ is $\CA$ itself.
\end{exam}

\begin{defn}[$E_2$-version] \label{defn:e2-mext}
Let $\CA$ be a braided fusion $n$-category over $\CE^{\op(-2)}$ where $\CE$ is an $E_3$-fusion $n$-category. A {\em modular extension} of $\CA$ is a nondegenerate braided fusion $n$-category $\CM$ equipped with a $\C$-linear braided monoidal functor $\CA\to\CM$ such that the composite $\C$-linear braided monoidal functor $\CA\boxtimes\CE\to\CA\to\CM$ is a perfect pairing. Two modular extensions $\CM$ and $\CN$ of $\CA$ are {\em equivalent} if there is a $\C$-linear braided monoidal equivalence $\CM\simeq\CN$ such that the composition $\CA\to\CM\simeq\CN$ is equivalent to $\CA\to\CN$.
\end{defn}


\begin{cor} \label{cor:mod-ext-e2}
The equivalence classes of modular extensions of $\CA$ one-to-one correspond to the equivalence classes of $\C$-linear braided monoidal equivalences $\overline{\FZ_1(\Sigma\CE)} \simeq \FZ_1(\Sigma\CA)$ such that the composition $\overline{\Sigma\CE} \to \overline{\FZ_1(\Sigma\CE)} \simeq \FZ_1(\Sigma\CA)$ is equivalent to $\Sigma\psi_\CA$.
\end{cor}
\begin{proof}
Note that $\Sigma\CA$ is a fusion $(n+1)$-category over $\overline{\Sigma\CE}$. According to Proposition \ref{prop:e2-centralizer}, giving a modular extension of $\CA$ is equivalent to giving an $E_1$-modular extension of $\Sigma\CA$. Therefore, the claim follows from Corollary \ref{cor:mod-ext-e1}.
\end{proof}

\begin{rem}
To summarize, modular extension is a special case of $E_1$-modular extension; $E_1$-modular extension is a special case of $E_0$-modular extension. 
It is possible to generalize the concept of modular extension to the case of $m > 2$ for $E_m$-fusion $n$-categories. However, it lacks a clear physical interpretation and no longer possesses the aforementioned desirable properties. For instance, the proof of Corollary \ref{cor:mod-ext-e2} relies on the technical Proposition \ref{prop:e2-centralizer}, which is difficult to extend to scenarios where $m > 2$.
\end{rem}

\begin{rem}
The notion of (minimal) modular extension was originally considered for braided fusion 1-categories \cite[Section 5.1]{Mu03b} and was further developed in \cite{LKW15,LKW17,BGHN+17,GVR17,VR19,DN20,Nik22,OY23,JFR24}. In fact, in the special case where $n=1$, if $\CM$ is a modular extension of $\CA$ then $\CE$ and $\CA$ are both full subcategories of $\CM$ by Example \ref{exam:centralizer-e2}. In particular, $\CE$ must be the M\"uger center $\FZ_2(\CA)$. Moreover, we have $\dim\CM = \dim\CA\cdot\dim\FZ_2(\CA)$ by \cite[Theorem 3.10]{DGNO10}, as formulated in \cite[Conjecture 5.2]{Mu03b}.
Modular extensions of higher braided fusion categories were first considered in \cite{KLWZZ20a}. 
\end{rem}

\begin{rem}
Corollary \ref{cor:mod-ext-e2} was conjectured in \cite{KLWZZ20a}. A proof in the case $n=1$ was obtained by Johnson-Freyd and Reutter \cite{JFR24} when this paper was in preparation.
\end{rem}

The following theorem generalizes an exact sequence for a symmetric fusion 1-category announced by Nikshych \cite{Nik22b}.

\begin{thm} \label{thm:mext-exseq}
Let $\CE$ be a nonzero multi-fusion $n$-category. There is an exact sequence
$$0\to \pi_{n+1}(\Sigma n\Vect)  \xrightarrow{\alpha_0} \pi_{n+1}(\Sigma\CE)  \xrightarrow{\beta_0}  \MExt^{E_0}(\Omega^n \CE) \xrightarrow{\gamma_0} \cdots $$
$$\xrightarrow{\gamma_{n-1}} \pi_1(\Sigma n\Vect)  \xrightarrow{\alpha_n} \pi_1(\Sigma\CE)  \xrightarrow{\beta_n}  \MExt^{E_0}(\CE)   \xrightarrow{\gamma_n}  \pi_0(\Sigma n\Vect)^\times \xrightarrow{\alpha_{n+1}} \pi_0(\Sigma\CE)
$$
where $\pi_i(\Sigma\CE)$ denotes the homotopy group of the underlying $(n+1)$-groupoid of $\Sigma\CE$.
\end{thm}
\begin{proof}
Let $\alpha_i$ be induced by the canonical monoidal functor $n\Vect\to\CE$. We identify $\MExt^{E_0}(\CE)$ with the group of $E_0$-modular extensions of $\CE^\op$ (see Remark \ref{rem:mext-e0-group}). 
Let $\gamma_n([\CM]) = [\CM]$. Then $\gamma_n$ is a group homomorphism as evident from the composite $\C$-linear equivalence $$(\CE^\op\boxtimes_{\FZ_0(\CE)}\CE) \boxtimes (\CE^\op\boxtimes_{\FZ_0(\CE)}\CE) \simeq \CE^\op\boxtimes_{\FZ_0(\CE)}\FZ_0(\CE)\boxtimes_{\FZ_0(\CE)}\CE \simeq \CE^\op\boxtimes_{\FZ_0(\CE)}\CE.$$

For an object $E\in\CE$, the composition $\CE^\op\boxtimes\CE \to \CE^\op \xrightarrow{E} n\Vect$ maps $X\boxtimes Y \mapsto \Hom_\CE(Y^L\otimes X,E)$, hence the induced $\C$-linear functor $\CE\to\FZ_0(\CE^\op,n\Vect)=\CE$ maps $Y\mapsto Y\otimes E$. In particular, $(n\Vect,\Hom_\CE(\one_\CE,E))$ is an $E_0$-modular extension of $\CE^\op$ if $E$ is invertible. We define $$\beta_n([E]) = [(n\Vect,\Hom_\CE(\one_\CE,E))].$$ To see that $\beta_n$ is a group homomorphism, we note that the induced monoidal equivalence $\FZ_0(\CE)\simeq\FZ_0(\CE)$ maps $F\mapsto F(-\otimes E)\otimes E^{-1}$. Reversing the construction, we obtain $\Ker\gamma_n=\Img\beta_n$, i.e. the sequence is exact at $\MExt^{E_0}(\CE)$. Moreover, we have $\beta_n([E]) = \beta_n([E'])$ if and only if the functors $E,E': \CE^\op \to n\Vect$ differ by a $\C$-linear auto-equivalence of $n\Vect$. Thus the sequence is exact at $\pi_1(\Sigma\CE)$.

If $\CM$ is an $E_0$-modular extensions of $\CE^\op$, then $\CE \simeq \FZ_0(\CE^\op,\CM) = \CE\boxtimes\CM$ as left $\CE$-modules, hence $\CE\simeq\CM^\op\boxtimes\CE$ as right $\CE$-modules by Proposition \ref{prop:dual-bimod}. Thus $\alpha_{n+1}\gamma_n([\CM]) = [\CM\boxtimes\CE] = [\CE]$. Conversely, if $\alpha_{n+1}([\CC]) = [\CE]$, then $\CE \simeq \CE\boxtimes\CC = \FZ_0(\CE^\op,\CC)$ as left $\CE$-modules, promoting $\CC$ to an $E_0$-modular extension of $\CE^\op$. Thus $\Ker\alpha_{n+1}=\Img\gamma_n$, i.e. the sequence is exact at $\pi_0(\Sigma n\Vect)^\times$.

Note that $\Ker\alpha_n$ is identical to the kernel of $\pi_0(\Sigma\Omega n\Vect)^\times \to \pi_0(\Sigma\Omega\CE)$. By induction on $n$, we conclude the theorem.
\end{proof}

\begin{rem}
Theorem \ref{thm:mext-exseq} may be proved by establishing a group isomorphism $$\MExt^{E_0}(\CE) \simeq \pi_1(\widetilde{\Sigma\CE}\sslash\widetilde{\Sigma n\Vect}{}^\times)$$ straightforwardly where $\widetilde{\Sigma\CE}$ is the underlying $(n+1)$-groupoid of $\Sigma\CE$. In particular, $\MExt^{E_0}(\CE)$ is an abelian group if $\CE$ is braided.
In fact, an element of $\pi_1(\widetilde{\Sigma\CE}\sslash\widetilde{\Sigma n\Vect}{}^\times)$ is represented by a right $\CE$-module equivalence $\CE\simeq\CC\boxtimes\CE$ where $\CC$ is an invertible object of $\Sigma n\Vect$, or, equivalently by Proposition \ref{prop:dual-bimod}, a left $\CE$-module equivalence $\CE\simeq\CE\boxtimes\CC^\op$. This induces a $\C$-linear monoidal equivalence $\FZ_0(\CE) \simeq \FZ_0(\CE\boxtimes\CC^\op) = \FZ_0(\CE)$ over $\CE$. Conversely, a $\C$-linear monoidal equivalence $\phi:\FZ_0(\CE) \simeq \FZ_0(\CE)$ over $\CE$ induces a left $\CE$-module equivalence $F:\CE \simeq \FZ_0(\CE)\boxtimes_{\FZ_0(\CE)}\CE = \CE\boxtimes(\CE^\op\boxtimes_{\FZ_0(\CE)}\CE)$. Note that the equivalence class of $F$ depends only on that of $F(\one_\CE)$ hence only on that of $\phi$.
\end{rem}

In the remainder of this subsection, we consider the special case of $E_0$-modular extensions of {\em local fusion $n$-categories} \cite{KLWZZ20b}, that is, fusion $n$-categories admitting fiber functors to $n\Vect$. This special case shares many desirable properties such as the fact that the extensions are always non-chiral, its classification is more explicit.

\begin{thm} \label{thm:sym-breaking}
Let $\CE$ be a fusion $n$-category which admits a $\C$-linear monoidal functor $F:\CE\to n\Vect$. Regard $\CE^\op$ as an $E_0$-fusion $n$-category over $\CE^\rev$. 

$(1)$ If $\CM$ is an $E_0$-modular extension of $\CE^\op$,
then the $\C$-linear functor $$F\circ G:\CM^\op\to n\Vect$$ is invertible where $G$ is the right adjoint functor to $\CE\to\CM^\op$.

$(2)$ The following three groups are canonically isomorphic to each other:
\begin{enumerate}
\renewcommand{\labelenumi}{$(\alph{enumi})$}
\item The group of equivalence classes of $E_0$-modular extensions of $\CE^\op$ (i.e. the group $\MExt^{E_0}(\CE)$, see Remark \ref{rem:mext-e0-group}).
\item The group of equivalence classes of invertible objects of $\CE$ modulo that of $n\Vect$.
\item The group of equivalence classes of invertible objects $E\in\CE$ such that $F(E)\simeq\one_{n\Vect}$.
\end{enumerate}
\end{thm}

\begin{proof}
(2) Since the composition $n\Vect\to\CE\xrightarrow{F}n\Vect$ is equivalent to the identity functor, the maps $\alpha_n,\alpha_{n+1}$ from Theorem \ref{thm:mext-exseq} are injective and hence $\gamma_n$ vanishes. We obtain a group isomorphism between $(a)$ and $(b)$. Since $F$ transforms invertible objects into invertible objects, we obtain a group isomorphism between $(b)$ and $(c)$. 

(1) Since $\gamma_n$ vanishes, we may choose a $\C$-linear identification $\CM = n\Vect$. Let $E\in\CE$ be the invertible object representing the defining $\C$-linear functor $\CE^\op\to\CM$. Then $G:\CM^\op\to\CE$ maps $V\mapsto V^\vee\boxtimes E$. Therefore, $F\circ G$ is invertible. 
\end{proof}

\begin{cor}
Let $\CE$ be a braided fusion $n$-category which admits a $\C$-linear braided monoidal functor $F:\CE\to n\Vect$. Regard $\CE^\rev$ as a fusion $n$-category over $\bar\CE$. If $\CM$ is an $E_1$-modular extension of $\CE^\rev$, then the evident $\C$-linear monoidal functor $$\CM^\rev \to \FZ_0(\CM\boxtimes_\CE n\Vect)$$ is invertible.
\end{cor}
\begin{proof}
The induced functor $\Sigma\CE\to\Sigma\CM^\rev$ is left adjoint to the forgetful functor $G:\Sigma\CM^\rev \to \Sigma\CE$. By Theorem \ref{thm:sym-breaking}(1), the $\C$-linear functor $\Sigma F\circ G:\Sigma\CM^\rev \to (n+1)\Vect$ is invertible. Note that $\Sigma F\circ G = -\boxtimes_\CE n\Vect$ which maps to $\CM$ to $\CM\boxtimes_\CE n\Vect$. Therefore, $\CM^\rev \simeq \FZ_0(\CM\boxtimes_\CE n\Vect)$.
\end{proof}

\begin{cor} \label{cor:sym-breaking2}
Let $\CE$ be an $E_3$-fusion $n$-category which admits a $\C$-linear $E_3$-monoidal functor $F:\CE\to n\Vect$. Regard $\bar\CE$ as a braided fusion $n$-category over $\CE^{\op(-2)}$. If $\CM$ is a modular extension of $\bar\CE$, then the evident $\C$-linear braided monoidal functor $$\bar\CM \to \FZ_1(\CM\boxtimes_\CE n\Vect)$$ is invertible.
\end{cor}

\begin{rem}
For a field $k$ that is not separably closed, Corollary \ref{cor:mod-ext-e2} needs an additional assumption that $\Omega^n\CE\simeq\Omega^n\CA\simeq k$, i.e. $\CE$ and $\CA$ remain of fusion type after base extensions. Otherwise, the proof of Corollary \ref{cor:mod-ext-e2} fails because a nondegenerate fusion $n$-category might be disconnected. For example, the braided fusion 0-category $\C$ over $\R$ admits no modular extension while the nondegenerate fusion 1-category $\Fun_\R(\Vect_\C,\Vect_\C)$ over $\R$ is disconnected. 
\end{rem}


\section{Group theoretical fusion $n$-categories} \label{sec:group-cat}

Group theoretical higher fusion categories represent a broad and largely unexplored subject, finding extensive application in the study of symmetries in topological phases of matter. In this section, we focus on two specific issues within this topic. In the first subsection, we provide precise definitions for $n\Rep\CG$ and $n\Vect_\CG$ of a finite higher group $\CG$, and prove that they are indeed fusion $n$-categories. In the second subsection, we investigate pointed fusion $n$-categories, which are natural extensions of pointed fusion 1-categories. Surprisingly, a fusion $n$-category $\CA$ for $n\ge2$ is pointed if and only if $\Omega\CA$ is trivial, contrasting drastically to the case $n=1$.

\subsection{Linear representations of higher groups} \label{rep-high-group}

Let $G$ be a (discrete) group, regarded as a monoidal 0-category. We use $n\Rep G$ to denote the condensation-complete $\C$-linear symmetric monoidal $n$-category $\LMod_G(n\Vect) = \Fun(B G,n\Vect)$. Note that $\Omega\Fun(B G,\CC)\simeq\Fun(B G,\Omega\CC)$ for any symmetric monoidal $n$-category $\CC$. In particular, $$\Omega (n+1)\Rep G\simeq n\Rep G.$$

We use $n\Vect_G$ to denote the evident condensation-complete $\C$-linear monoidal $n$-category $\bigoplus_{g\in G}n\Vect$ whose tensor product is induced by the product of $G$ (a formal definition is given below).  
Note that 
$$\LMod_{n\Vect_G}((n+1)\Vect) \simeq (n+1)\Rep G,$$
$$\Fun_{n\Vect_G}(n\Vect,n\Vect) \simeq \Omega(n+1)\Rep G.$$

\begin{prop} \label{prop:nvecg-point}
Let $G$ be a finite group and $n\ge1$ be an integer. $(1)$ $n\Vect_G$ is a fusion $n$-category. $(2)$ $n\Rep G$ is a symmetric fusion $n$-category and we have $(n+1)\Rep G \simeq \Sigma n\Rep G$.
\end{prop}
\begin{proof}
(1) Let $\CA=n\Vect_G$, $\CR=\RMod_\CA((n+1)\Vect)$ and $\CB=\BMod_{\CA|\CA}((n+1)\Vect)$. Let $v_0$ denote the tensor product functor $\otimes:\CA\boxtimes\CA\to\CA$ and let $v_{i+1}:v_i\circ v_i^R\to\Id_{\cdots\Id_\CA}$ be the iterated counit map in $\CR$, $0\le i<n$. 
Note that 
the equivalence $\Fun_{\CA^\rev}(\CA,\CA) \simeq \CA$ transforms $\Id_\CA$, $v_0\circ v_0^R$ and $v_1$ into $\one_\CA$, $\oplus_{g\in G}\one_\CA$ and the folding map, respectively. Thus $v_i^R$ and $v_{i+1}$ are transformed into the diagonal map and the folding map, respectively, by induction on $1\le i<n$. In particular, 
$\Hom_\CR(\Id,v_{n-1}\circ v_{n-1}^R) \simeq \Hom_\CR(v_{n-1}\circ v_{n-1}^R,\Id) \simeq \oplus_{g\in G}\C$.

Lift the counit maps $v_i$ to $\CB$. Note that $\Fun_{\CA|\CA}(\CA,\CA)$ is the homotopic $G$-invariants in $\Fun_{\CA^\rev}(\CA,\CA)$. Thus $\Hom_\CB(\Id,v_{n-1}\circ v_{n-1}^R)$ and $\Hom_\CB(v_{n-1}\circ v_{n-1}^R,\Id)$ are the $G$-invariants in $\Hom_\CR(\Id,v_{n-1}\circ v_{n-1}^R)$ and $\Hom_\CR(v_{n-1}\circ v_{n-1}^R,\Id)$, respectively. Observe that $G$ acts on $v_0\circ v_0^R$ by translating its direct summands. Therefore, there exists a unique morphism $u:\Id\to v_{n-1}\circ v_{n-1}^R$ in $\CB$ such that $v_n\circ u=|G|$.

The morphisms $v_0^R,v_1,v_1^R,\dots,v_n$ extend the tensor product functor $\otimes:\CA\boxtimes\CA\to\CA$ to a condensation in $\CB$, terminated by the identity $v_n\circ \frac1{|G|}u=1$. We obtain a condensation $-\boxtimes\CA\condense\Id_{\Sigma\CA}$, which extends $-\boxtimes\CA:(n+1)\Vect\to\Sigma\CA$ to a condensation. This shows that $\Sigma\CA$ is separable hence $\CA$ is a fusion $n$-category.

(2) Since $n\Vect_G$ is a fusion $n$-category, $(n+1)\Rep G$ is connected. Then from $\Omega (n+1)\Rep G\simeq n\Rep G$ we obtain $(n+1)\Rep G \simeq \Sigma n\Rep G$. Invoking \cite[Theorem 3.35]{KZ22}, we conclude that $n\Rep G$ is a symmetric fusion $n$-category.
\end{proof}

\begin{defn}
An {\em $n$-groupoid} $\CC$ is an $n$-category where all (higher) morphisms are invertible. An {\em $n$-group} is a monoidal $(n-1)$-category $\CG$ such that $B\CG$ is an $n$-groupoid. We say that $\CG$ is {\em finite} if all the homotopy groups $\pi_i(\CG)$ are finite. A {\em group homomorphism} between two $n$-groups is a monoidal functor.
\end{defn}

\begin{rem}
A 1-group is simply a group. An $n$-group can be regarded as an $m$-group for $m>n$. We adopt the convention that a 0-group is a trivial group.
\end{rem}

For an $m$-group $\CG$, we use $n\Rep\CG$ to denote the condensation-complete $\C$-linear symmetric monoidal $n$-category $\LMod_\CG(n\Vect) = \Fun(B\CG,n\Vect)$. Again, $$\Omega (n+1)\Rep \CG\simeq n\Rep\CG.$$ Moreover, $n\Rep\CG\simeq n\Rep(\tau_{\le n-1}\CG)$, where $\tau_{\le n-1}\CG$ is the truncation of $\CG$ (i.e. $\pi_i(\tau_{\le n-1}\CG) = \pi_i(\CG)$ for $i\le n-1$ and vanishes for $i\ge n$).

The forgetful functor $\KarCat^\C_{n+1} \to \KarCat_{n+1}$ has a left adjoint\footnote{
A formal argument shows that the subcategory $\Cat^+_{n+1} \subset \Cat_{n+1}$ of additive categories and additive functors (see \cite[Subsection 2.2]{KZ22}) is closed under limits and filtered colimits. For example, in a limit category $\varprojlim\CA_\alpha$, an object $X$ is represented by a coherent family of objects $X_\alpha\in\CA_\alpha$, and the direct sum $X\oplus Y$ is represented by the coherent family $X_\alpha\oplus Y_\alpha\in\CA_\alpha$. In a filtered colimit category $\varinjlim\CA_\alpha$, an object $X$ is the image of an object $X_\alpha\in\CA_\alpha$ for sufficiently large $\alpha$, and the direct sum $X\oplus Y$ is that of $X_\alpha\oplus Y_\alpha\in\CA_\alpha$. Invoking the adjoint functor theorem (c.f. \cite[Corollary 5.5.2.9]{Lur09}) which we assume in a concrete model, we conclude that the inclusion functor $\Cat^+_{n+1} \to \Cat_{n+1}$ admits a left adjoint $L':\Cat_{n+1} \to \Cat^+_{n+1}$. 
Moreover, the forgetful functor $\Cat^\C_{n+1} = \LMod_{B^{n+1}\C}(\Cat^+_{n+1}) \to \Cat^+_{n+1}$ admits a left adjoint $L''=B^{n+1}\C\boxtimes-$. The composition $\KarCat_{n+1} \hookrightarrow \Cat_{n+1} \xrightarrow{L'} \Cat^+_{n+1} \xrightarrow{L''} \Cat^\C_{n+1} \xrightarrow{\Kar} \KarCat^\C_{n+1}$ is the desired left adjoint $L$.
}
$$L:\KarCat_{n+1} \to \KarCat^\C_{n+1}.$$ 
For an $(n+1)$-group $\CG$, we use $n\Vect_\CG$ to denote the condensation-complete $\C$-linear monoidal $n$-category $\Omega L(B\CG)$.

\begin{rem} \label{rem:VecG-decomp}
Since $L$ is symmetric monoidal, it induces a symmetric monoidal functor $E_m\KarCat_{n+1} \to E_m\KarCat^\C_{n+1}$ for $m\ge0$ that is left adjoint to the forgetful functor. Thus we have a $\C$-linear equivalence $n\Vect_\CG \simeq L(\CG)$. That is, $$n\Vect_\CG \simeq \bigoplus_{g\in\pi_0(\CG)} \Sigma(n-1)\Vect_{\Omega\CG}.$$
It follows immediately that $n\Vect_\CG$ is a fusion $n$-category only if $\CG$ is finite.
In the special case where $\CG$ is a group $G$, $n\Vect_\CG$ is precisely $n\Vect_G$ defined above.
\end{rem}

\begin{prop} \label{prop:nVec-RepG}
For an $(n+1)$-group $\CG$, we have
$$\LMod_{n\Vect_\CG}((n+1)\Vect) \simeq (n+1)\Rep\CG,$$
$$\Fun_{n\Vect_\CG}(n\Vect,n\Vect) \simeq \Omega(n+1)\Rep\CG.$$ 
\end{prop}
\begin{proof}
By definition, we have $\LMod_{n\Vect_\CG}((n+1)\Vect) = \Fun_\C(L(B\CG),(n+1)\Vect)$ and $(n+1)\Rep\CG = \Fun(B\CG,(n+1)\Vect)$. The former equivalence follows. The latter equivalence follows from the former one.
\end{proof}

\begin{rem}
$n\Rep\CG$ might be a symmetric fusion $n$-category for infinite $\CG$. For example, $\Rep S_\infty \simeq \Rep\Z_2$ where $S_\infty = \varinjlim S_n$ is the infinite symmetric group.
\end{rem}

\begin{rem}
Note that $L(\CC\times\CD) \simeq L(\CC)\boxtimes L(\CD)$. Thus $n\Vect_{\CG_1\times\CG_2} \simeq n\Vect_{\CG_1}\boxtimes n\Vect_{\CG_2}$. 
\end{rem}

\begin{exam} \label{exam:em-nvecg}
Let $G$ be a finite abelian group. Then $B^{n-1}G$ is a finite $n$-group for $n\ge1$. By induction on $n$, we have $n\Vect_{B^{n-1}G} \simeq n\Rep\chi(G)$ where $\chi(G)=\hom(G,\C^\times)$ is the character group. It follows that $n\Rep(B^{n-1}G) \simeq n\Vect_{\chi(G)}$ 
\end{exam}

\begin{thm} \label{thm:nrepg-fus}
Let $\CG$ be a finite $n$-group. $(1)$ $n\Vect_\CG$ is a fusion $n$-category. $(2)$ $n\Rep\CG$ is a symmetric fusion $n$-category and we have $(n+1)\Rep\CG \simeq \Sigma n\Rep\CG$.
\end{thm}

\begin{proof}
(1) We prove by induction on $n$. The case $n=0$ is trivial. For $n\ge1$, let $\CH=B\Omega\CG$. Then $n\Vect_\CG$ has the form $\bigoplus_{g\in\pi_0(\CG)} n\Vect_\CH$ by Remark \ref{rem:VecG-decomp} where $n\Vect_\CH \simeq \Sigma(n-1)\Vect_{\Omega\CG}$ is a fusion $n$-category by the inductive hypothesis. Then a similar argument of $\CG$-invariants as the proof of Proposition \ref{prop:nvecg-point} shows that the tensor product $\otimes: n\Vect_\CG\boxtimes_{n\Vect_\CH}n\Vect_\CG \to n\Vect_\CG$ extends to a condensation in $\BMod_{n\Vect_\CG|n\Vect_\CG}((n+1)\Vect)$ hence $-\boxtimes_{n\Vect_\CH}n\Vect_\CG: \Sigma n\Vect_\CH \to \Sigma n\Vect_\CG$ extends to a condensation. Therefore, $n\Vect_\CG$ is a fusion $n$-category.

(2) The proof is the same as Proposition \ref{prop:nvecg-point}.
\end{proof}

\begin{cor}
Let $\CG$ be a finite $(n+1)$-group $\CG$. $(1)$ $n\Vect_\CG$ is a multi-fusion $n$-category. $(2)$ $n\Vect_\CG$ is of fusion type if and only if $\pi_n(\CG)$ is trivial, i.e. $\CG$ is an $n$-group. 
\end{cor}
\begin{proof}
(1) is because $\LMod_{n\Vect_\CG}((n+1)\Vect)\simeq(n+1)\Rep\CG$ is separable.
(2) can be seen by induction on $n$ and from the base case that the group algebra $0\Vect_G$ of a finite group $G$ is decomposable unless $G$ is trivial.
\end{proof}

\begin{rem}
In the $k$-linear settings where $k$ is a field of positive characteristic, the results of this and next subsections need an additional assumption that $\chara k$ does not divide the order of $G$ or none of the order of $\pi_i(\CG)$.
\end{rem}

\subsection{Pointed fusion $n$-categories} \label{sec:point-fus}

\begin{defn}
Let $\CA$ be a fusion $n$-category and let $\CA=\bigoplus_{g\in\Lambda}\CA_g$ be the simple decomposition of the underlying separable $n$-category. We say that $\CA$ is {\em pointed} if $\Lambda$ admits the structure of a group such that the tensor product $\otimes:\CA\boxtimes\CA\to\CA$ induces an equivalence $\CA_g\boxtimes\CA_h \simeq \CA_{g h}$ for all $g,h\in\Lambda$. We refer to $\Lambda$ as the {\em fusion group} of $\CA$.
\end{defn}

\begin{exam} \label{exam:point-fusion}
(1) Let $G$ be a finite group. A monoidal functor $\rho:G\to(n+1)\Vect$ induces a monoidal functor $\tilde\rho:(n+1)\Vect_G\to(n+1)\Vect$. Since $\tilde\rho$ extends to a unital condensation of right $(n+1)\Vect_G$-modules, $\tilde\rho^R(n\Vect)$ lifts to a unital condensation algebra $A$ in $(n+1)\Vect_G$. The image $A'$ of $A$ in $(n+1)\Vect$ determines a multi-fusion $n$-category $\CA$ such that $\Sigma\CA$ is the $E_0$-multi-fusion $(n+1)$-category of condensation right $A'$-modules. Note that $\CA \simeq \bigoplus_{g\in G}\rho(g)$ and the tensor product of $\CA$ is induced by the natural equivalence $\rho(g)\boxtimes\rho(h)\simeq\rho(g h)$. Therefore, $\CA$ is pointed with fusion group $G$ for $n\ge1$.

(2) Let $\alpha: B G\to B^{n+2}\C^\times$ be a functor that represents a cohomology class $\omega\in H^{n+2}(G,\C^\times)$. Applying the above construction to the composition $G\xrightarrow{\Omega\alpha}B^{n+1}\C^\times \subset (n+1)\Vect$ yields a multi-fusion $n$-category which we denote by $n\Vect_G^\omega$. 
Informally speaking, $n\Vect_G^\omega$ is obtained from $n\Vect_G$ by twisting the coherence relation.
\end{exam}

\begin{prop} \label{prop:point-fus}
Let $\CA$ be a pointed fusion $n$-category with fusion group $G$. We have the following assertions:
\begin{enumerate}
\item $\CA_e$ is a trivial fusion $n$-category where $e\in G$ is the unit element.
\item $\CA_g$ is invertible for all $g\in G$.
\item $X\otimes Y$ is simple for any simple objects $X,Y\in\CA$.
\end{enumerate}
\end{prop}
\begin{proof}
(1) Since $\CA_e\boxtimes\CA_e\simeq\CA_e$, $\CA_e\simeq n\Vect$ by Proposition \ref{prop:sep-idem}. By Proposition \ref{prop:fus-trivial}, $\CA_e$ is trivial.

(2) is a consequence of (1). 

(3) is a consequence of \cite[Lemma 3.49(1)]{KZ22}.
\end{proof}

\begin{thm} \label{thm:mext-repG}
Let $G$ be a finite group. The following three types of data one-to-one correspond to each other:
\begin{enumerate}
\renewcommand{\labelenumi}{$(\alph{enumi})$}
\item Equivalence classes of modular extensions of $n\Rep G$ (regarded as a braided fusion $n$-category over itself).
\item Equivalence classes of invertible objects of $(n+2)\Rep G$ whose image in $(n+2)\Vect$ is equivalent to $(n+1)\Vect$.
\item Equivalence classes of pointed fusion $n$-categories with fusion group $G$.
\end{enumerate}
\end{thm}
\begin{proof}
$(a)\Leftrightarrow(b)$ is due to Theorem \ref{thm:sym-breaking}(2).

$(b)\Leftrightarrow(c)$ Let $X\in(n+2)\Rep G$ be an invertible object whose image in $(n+2)\Vect$ is $(n+1)\Vect$. Then $X$ determines a monoidal functor $\rho:G\to(n+1)\Vect$ which induces a pointed fusion $n$-category with fusion group $G$ by Example \ref{exam:point-fusion}(1). 

Conversely, let $\CA$ be a pointed fusion $n$-category with fusion group $G$. Then $\CA$, viewed as a unital condensation algebra in $(n+1)\Vect$, lifts to a unital condensation algebra $A$ in $(n+1)\Vect_G$ which defines an object $X\in(n+2)\Rep G$. The image of $X$ in $(n+2)\Vect$ is given by $\Hom_{(n+2)\Rep G}((n+1)\Vect_G,X)$, i.e. the category of condensation right $A$-modules. According to Proposition \ref{prop:point-fus}(1)(2), the canonical map $M_e\boxtimes A\to M$ for a condensation right $A$-module $M=\oplus_{g\in G}M_g$ is invertible. Therefore, the image of $X$ in $(n+2)\Vect$ is equivalent to $(n+1)\Vect$. 

These two constructions are clearly inverse to each other.
\end{proof}

\begin{rem} \label{rem:mext-repG}
Unwinding the constructions, we see that the three types of data from Theorem \ref{thm:mext-repG} correspond in the following way. A pointed fusion $n$-category $\CA$ with fusion group $G$ induces an invertible object $X\in(n+2)\Rep G$ as well as a unital condensation algebra $A$ in $(n+1)\Vect_G$. We have $\Hom_{(n+2)\Rep G}(\one,X) \simeq \Sigma\CA$ because the left hand side is the category of condensation right $A$-modules in $\Hom_{(n+2)\Rep G}(\one,(n+1)\Vect_G) \simeq (n+1)\Vect$. Then $X$ induces an $E_0$-modular extension of $((n+2)\Rep G)^\op$
$$((n+2)\Rep G)^\op \to ((n+2)\Vect,\Sigma\CA), \quad Y\mapsto\Hom_{(n+2)\Rep G}(Y,X).$$
This yields a modular extension of $\overline{n\Rep G}$ given by $\Omega^2((n+2)\Vect,\Sigma\CA) = \Omega\FZ_0(\Sigma\CA) = \FZ_1(\CA)$. 

Moreover, unwinding the proof of Corollary \ref{cor:sym-breaking2}, we see that $$\FZ_1(\CA)\boxtimes_{n\Rep G}n\Vect \simeq \CA^\rev.$$ 
Thus the fiber $\FZ_1(\CA)\times_\CA\CA_e$ is the trivial $n\Rep G$-module. Therefore, the braided monoidal functor $\overline{n\Rep G}\to\FZ_1(\CA)$ is fully faithful.
\end{rem}

\begin{prop} \label{prop:point-simple}
Let $\CA$ be a fusion $n$-category and let $\CA=\bigoplus_{g\in\Lambda}\CA_g$ be the simple decomposition of the separable $n$-category. The following conditions are equivalent:
\begin{enumerate}
\item $\CA$ is pointed.
\item $\Omega\CA$ is trivial and $X^L\otimes X$ is simple for any simple $X\in\CA$.
\item $\Omega\CA$ is trivial and $\Lambda$ can be promoted to a group such that $\CA_g\otimes\CA_h \subset \CA_{g h}$.
\end{enumerate}
\end{prop}

\begin{proof}
$(1)\Rightarrow(2)$ is immediate from Proposition \ref{prop:point-fus}(1)(3).

$(2)\Rightarrow(3)$ If $X,Y\in\CA$ are simple, then $(X^L\otimes X)\otimes(Y\otimes Y^R)$ is simple by \cite[Lemma 3.49(1)]{KZ22}. By Proposition \ref{prop:fus-nonzero}, $X\otimes Y$ has to be simple. Therefore, the tensor product of $\CA$ promotes $\Lambda$ to a group such that $\CA_g\otimes\CA_h \subset \CA_{g h}$. 

$(3)\Rightarrow(1)$
We need to show that the tensor product functor $v:\CA_g\boxtimes\CA_{g^{-1}}\to\CA_e$ is invertible for all $g\in\Lambda$. Indeed, by the duality of $\CA$, $v$ exhibits the separable $n$-category $\CA_{g^{-1}}$ dual to $\CA_g$. 
The associativity of the tensor product functor $\CA_g\boxtimes\CA_{g^{-1}}\boxtimes\CA_g\to\CA_g$ then implies that $v$ and the unit map $u: n\Vect \to \CA_g\boxtimes\CA_{g^{-1}}$ exhibit $\CA_g\boxtimes\CA_{g^{-1}}$ as a retract of $n\Vect$. Invoking Corollary \ref{cor:ret-nvec} we conclude that $v$ is invertible, as desired.
\end{proof}

\begin{thm} \label{thm:point-fus}
Let $\CA$ be a fusion $n$-category such that $\Omega\CA$ is trivial where $n\ge2$. Then $\CA$ is pointed.
\end{thm}
\begin{proof}
Let $X\in\CA$ be a simple object. The equivalence $\Hom_\CA(X,X) \simeq \Hom_\CA(\one_\CA,X\otimes X^L)$ maps $\Id_X$ to the unit map $u:\one_\CA\to X\otimes X^L$. Thus $u$ is simple. Note that $u$ can be regarded as a morphism of $\Sigma\Omega\CA\simeq n\Vect$. The morphism $u\otimes\Id_X: X\to X\otimes X^L\otimes X$ is simple by \cite[Lemma 3.49(2)]{KZ22}. Therefore, the induced morphism $\Id_{X^L\otimes X}$ is also simple. Namely, $X^L\otimes X$ is simple. Invoking Proposition \ref{prop:point-simple} we conclude that $\CA$ is pointed.
\end{proof}

\begin{cor} \label{cor:pointed-z1}
Let $\CA$ be a fusion $n$-category where $n\ge2$. Then $\CA$ is trivial if and only if $\Omega\CA$ and $\Omega^{n-1}\FZ_1(\CA)$ are both trivial.
\end{cor}
\begin{proof}
Suppose that $\Omega\CA$ is trivial. By Theorem \ref{thm:point-fus}, $\CA$ is pointed. Then by Remark \ref{rem:mext-repG}, $\Omega^{n-1}\FZ_1(\CA) \simeq \Rep G$ where $G$ is the fusion group of $\CA$. Therefore, if $\Omega^{n-1}\FZ_1(\CA)$ is trivial, then $\CA$ is trivial.
\end{proof}

\begin{cor} \label{cor:nondeg-trivial}
Let $\CA$ be a nondegenerate braided fusion $n$-category where $n\ge2$. Then $\CA$ is trivial if and only if $\Omega\CA$ is trivial.
\end{cor}
\begin{proof}
Suppose that $\Omega\CA$ is trivial. Since $\CA$ is nondegenerate, $\Omega\FZ_1(\CA) \simeq \Omega(\CA\boxtimes\bar\CA)$ is trivial. Then by Corollary \ref{cor:pointed-z1}, $\CA$ is trivial.
\end{proof}

\begin{prop} \label{prop:nondeg-bfc2}
The following conditions are equivalent for a braided fusion $n$-category $\CA$:
\begin{enumerate}
\item[(5)] $\FZ_1(\Sigma\CA)$ is trivial.
\item[(6)] $\FZ_2(\CA)$ is trivial.
\end{enumerate}
\end{prop}

\begin{proof}
$(5)\Rightarrow(6)$ $\FZ_2(\CA) = \Omega\FZ_1(\Sigma\CA)$ is trivial.
$(6)\Rightarrow(5)$ Assume $n\ge1$. Let $\CB=\FZ_1(\Sigma\CA)$ so that $\Omega\CB$ is trivial. By Corollary \ref{cor:z1-nondeg}, $\CB$ is nondegenerate. Then by Corollary \ref{cor:nondeg-trivial}, $\CB$ is trivial.
\end{proof}

\begin{rem}
For a pointed fusion $n$-category over an arbitrary field $k$, $\chara k$ may not divide the order of the fusion group. Otherwise, we follow the lines of the proof of Proposition \ref{prop:nvecg-point} and draw a contradiction.
\end{rem}

\begin{rem}
Proposition \ref{prop:point-fus}(3), Proposition \ref{prop:point-simple} and Theorem \ref{thm:point-fus} 
only apply to separably closed base fields because their proofs rely essentially on \cite[Lemma 3.49]{KZ22}. Corollary \ref{cor:pointed-z1}, Corollary \ref{cor:nondeg-trivial} and Proposition \ref{prop:nondeg-bfc2} apply to an arbitrary base field because the claims can be reduced to the separably closed case.
\end{rem}

\begin{rem}
Theorem \ref{thm:point-fus} for $n=2$ is due to \cite[Theorem A]{JFY21}.
Corollary \ref{cor:nondeg-trivial} for $n=2$ was predicted in \cite{LKW18}.
Corollary \ref{cor:nondeg-trivial} and Proposition \ref{prop:nondeg-bfc2} are due to \cite{JF22}.
Direction $(1)\Leftrightarrow(6)$ of Proposition \ref{prop:nondeg-bfc} and \ref{prop:nondeg-bfc2} for $n=1$ is due to \cite[Proposition 3.7]{DGNO10}.
\end{rem}

\section{Orthogonal higher categories} \label{sec:orth-cat}

The objective of this section is to outline a theory of orthogonal higher categories and to examine the orthogonal higher representations of finite groups. Our motivation arises from a desire to encapsulate, within a mathematical framework, properties of anti-unitary symmetries, such as the time-reversal symmetry, observed in topological phases of matter. The exposition of the first subsection is parallel to \cite[Subsection 4.2]{KZ22} where a theory of unitary higher categories has been outlined. The results from the second subsection will be applied in Subsection \ref{sec:anti-unitary-symm} to classify gapped quantum liquids with anti-unitary symmetries.

\medskip

\subsection{Orthogonal $n$-categories}

An {\em $\R^*$-$n$-category} is an $\R$-linear $n$-category $\CC$ equipped with an $\R$-linear involution $*:\CC\to\CC^{\op n}$ which fixes all the objects and all the $k$-morphisms for $k<n$. 
An {\em $\R^*$-functor} $F:\CC\to\CD$ between two $\R^*$-$n$-categories is a $*$-equivariant $\R$-linear functor. Similarly, a {\em (higher) $\R^*$-natural transformation} is a $*$-equivariant $\R$-linear (higher) natural transformation. 

Let $\Cat^{\R*}_n$ denote the $(n+1)$-category formed by $\R^*$-$n$-categories, $\R^*$-functors and (higher) $\R^*$-natural transformations and let $\KarCat^{\R*}_n$ denote the full subcategory of $*$-condensation-complete $\R^*$-$n$-categories.
By slightly abusing notation, we use $\Fun_\R(\CC,\CD)$ to denote $\Hom_{\Cat^{\R*}_n}(\CC,\CD)$.
We use $\boxtimes_\R$ to denote the tensor product of $\KarCat^{\R*}_n$ and leave the notation $\boxtimes$ for $\KarCat^*_n$.

The theory of $\R^*$-$n$-categories is completely parallel to that of $\R$-linear $n$-categories and that of unitary $n$-categories with some new features. Several definitions and results are listed as follows. C.f. \cite[Subsection 4.2]{KZ22}.

\smallskip

We use $\Hilb_\R$ to denote the symmetric monoidal $\R^*$-1-category of finite-dimensional real Hilbert spaces, and
use $n\Hilb_\R$ to denote the symmetric monoidal $\R^*$-$n$-category $\Sigma_*^{n-1}\Hilb_\R = \Sigma_*^n\R$.

\begin{prop}
The functor $\Hom_{(n+1)\Hilb_\R}(\bullet,-): (n+1)\Hilb_\R \to \Cat^{\R*}_n$ is fully faithful.
\end{prop}

\begin{defn} \label{defn:orth-cat}
An {\em orthogonal $n$-category} is an $\R^*$-$n$-category that lies in the essential image of the above functor.
\end{defn}

\begin{rem}
Since $n\Hilb_\R$ is an iterated delooping of $\Hilb_\R$, it is {\em positive}: $f^*\circ f\neq0$ for any nonzero $n$-morphism $f$. Since orthogonal $n$-categories are hom categories of $(n+1)\Hilb_\R$, they are also positive.
\end{rem}

\begin{prop}
Let $\CC$ be an orthogonal $n$-category. Then $\Hom_\CC(A,B)$ is an orthogonal $(n-1)$-category for any two objects $A,B\in\CC$.
\end{prop}

\begin{defn} 
An {\em orthogonal $E_m$-multi-fusion $n$-category} is an $E_m$-monoidal $\R^*$-$n$-category $\CA$ such that $\Sigma_*^m\CA$ is an orthogonal $(n+m)$-category. An orthogonal $E_m$-multi-fusion $n$-category with a simple tensor unit is also referred to as an {\em orthogonal $E_m$-fusion $n$-category}.
\end{defn}

\begin{thm}
Let $\CA$ be an $E_m$-monoidal $\R^*$-$n$-category where $n\ge1$. The following conditions are equivalent:
\begin{enumerate}
\item $\CA$ is an orthogonal $E_m$-multi-fusion $n$-category.
\item $\CA$ has duals and its underlying $\R^*$-$n$-category is orthogonal.
\end{enumerate}
\end{thm}

\begin{exam} \label{exam:orth-cat}
(1) Since $\C$ is an orthogonal symmetric fusion 0-category, $n\Hilb$ is an orthogonal symmetric fusion $n$-category. Therefore, unitary $n$-categories are orthogonal $n$-categories. In this sense, unitary theory is covered by orthogonal theory. 

(2) There are a total of three orthogonal fusion 0-categories: $\R$, $\C$ and the quaternion algebra $\BH$. Therefore, there are a total of three indecomposable orthogonal 1-categories: $\Hilb_\R$, $\Hilb$ and the $\R^*$-1-category of finite-dimensional quaternionic Hilbert spaces $\Hilb_\BH$. 

(3) We have isomorphisms of $\R$-algebras 
$$\C\otimes_\R\C \simeq \C\oplus\C, \quad \C\otimes_\R\BH \simeq \hom_\C(\C^2,\C^2), \quad \BH\otimes_\R\BH \simeq \hom_\R(\R^4,\R^4).$$ 
Thus
$$\Hilb \boxtimes_\R \Hilb \simeq \Hilb\oplus\Hilb, \quad
  \Hilb \boxtimes_\R \Hilb_\BH \simeq \Hilb, \quad
  \Hilb_\BH \boxtimes_\R \Hilb_\BH \simeq \Hilb_\R.
$$
Therefore, the orthogonal 1-categories $\Hilb_\R$ and $\Hilb_\BH$ are invertible, but $\Hilb$ is not.
\end{exam}

\begin{rem}
The inclusion $\iota:\R\to\C$ induces an orthogonal symmetric monoidal functor
$$\Sigma_*^{n+1}\iota: (n+1)\Hilb_\R \to (n+1)\Hilb, \quad \CC \mapsto \CC^\C := \CC\boxtimes_\R n\Hilb.$$
It is left adjoint to the restriction functor $(n+1)\Hilb \to (n+1)\Hilb_\R$. We refer to $\CC^\C$ as the {\em complexification} of $\CC$ and refer to $\CC$ as a {\em real form} of $\CC^\C$. For example, $\Hilb_\BH$ is a real form of $\Hilb$.
\end{rem}

\begin{rem}
A simple object $X\in\CC$ of an orthogonal 1-category describes a {\em time-reversal singlet} in the physics literature if $\Hom_\CC(X,X)=\R$; a {\em time-reversal doublet} if $\Hom_\CC(X,X)=\C$; a {\em Kramer doublet} if $\Hom_\CC(X,X)=\BH$. According to Example \ref{exam:orth-cat}(3), a time-reversal singlet remains simple in $\CC^\C$; a time-reversal doublet breaks into a direct sum of two distinct simple objects; a Kramer doublet breaks into a direct sum of two same simple objects.
\end{rem}

\subsection{Orthogonal representations of finite groups}

Let $\Z_2^T$ denote the Galois group $\Gal(\C/\R) = \{1,\dagger\}$ where $\dagger:\C\to\C$ is the complex conjugation. For a finite group $G$ equipped with a group homomorphism $\pi:G\to\Z_2^T$, the {\em twisted group algebra} $\C[G^\pi]$ is defined to be the $\R$-algebra obeying the multiplication rule
$$g\cdot h = g h, \quad g\cdot\lambda = \pi(g)(\lambda)g, \qquad g,h\in G, \; \lambda\in\C.$$
We use $\Rep G^\pi$ to denote the orthogonal symmetric fusion 1-category $\LMod_{\C[G^\pi]}(\Hilb_\R)$ where the tensor product of two modules $X$ and $Y$ is $X\otimes_\C Y$. We define $$n\Rep G^\pi = \Sigma_*^{n-1}\Rep G^\pi.$$

\begin{exam} \label{exam:rep-Gpi}
(1) If $\pi:G\to\Z_2^T$ is trivial, i.e. $\pi(G)=\{1\}$, then $\C[G^\pi]$ coincides with the usual group algebra $\C[G]$ and $n\Rep G^\pi$ coincides with $n\Rep G = \LMod_G(n\Hilb)$.

(2) We have $\C[\Z_2^T] \simeq \hom_\R(\C,\C)$ where $\Z_2^T$ is equipped with the identity homomorphism onto itself. Thus $n\Rep\Z_2^T \simeq n\Hilb_\R$.

(3) We have $\C[H\times\Z_2^T] \simeq \hom_\R(\C,\C) \otimes_\R \R[H]$ thus $n\Rep(H\times\Z_2^T) \simeq n\Rep_\R H=\LMod_H(n\Hilb_\R)$ where $H\times\Z_2^T$ is equipped with the projection onto the second variable.
\end{exam}

More generally, one can assign a super element, i.e. a central element $z\in G$ of order two such that $\pi(z)=1$, to define an orthogonal symmetric fusion 1-category $\sRep G^\pi$ where the braiding is twisted by $z$. As before, let $n\sRep G^\pi = \Sigma_*^{n-1}\sRep G^\pi$.
For example, $\sRep\Z_2^f$, where $\Z_2^f$ denotes the super $\Z_2$, is the unitary symmetric fusion 1-category of finite-dimensional super Hilbert spaces.
Similarly, $\sRep(\Z_2^f\times\Z_2^T) \simeq \sRep_\R\Z_2^f$ is the orthogonal symmetric fusion 1-category of finite-dimensional super real Hilbert spaces.

\begin{prop} \label{prop:complex-RepG}
Let $G$ be a finite group and let $\pi:G\to\Z_2^T$ be a nontrivial homomorphism. Then the restriction functor $n\Rep G^\pi\to n\Rep(\Ker\pi)$ exhibits $n\Rep(\Ker\pi)$ as the complexification of $n\Rep G^\pi$. If, in addition, $G$ is a super group, then the restriction functor $n\sRep G^\pi\to n\sRep(\Ker\pi)$ exhibits $n\sRep(\Ker\pi)$ as the complexification of $n\sRep G^\pi$.
\end{prop}
\begin{proof}
We have a decomposition of left $\C[\Ker\pi]$-modules 
$$\C[G^\pi]=\C[\Ker\pi]\oplus\C[\Ker\pi]g$$ 
where $g\in\pi^{-1}(\dagger)$ hence an isomorphism of $\C$-algebras 
$$\phi: \C[G^\pi]\otimes_\R\C \to \hom_{\C[\Ker\pi]}(\C[G^\pi],\C[G^\pi])^\rev, \quad \phi(x\otimes_\R\lambda)(y)=\lambda y x.$$ 
Therefore, the $\C[\Ker\pi]$-$\C[G^\pi]\otimes_\R\C$-bimodule $\C[G^\pi]$ is invertible hence induces a $*$-equivalence 
$$\Rep G^\pi\boxtimes_\R\Hilb \simeq \Rep(\Ker\pi), \quad X\otimes_\R\C\mapsto X.$$ 
Endowing the $*$-equivalence with the untwisted or the twisted symmetric monoidal structure and delooping for $n-1$ times then yield the conclusion.
\end{proof}

\begin{rem} \label{rem:real-form}
Note that giving a real Hilbert space $X\in\Hilb_\R$ is equivalent to giving a complex Hilbert space $Y\in\Hilb$ carrying an anti-unitary involution (indeed, $Y=X\otimes_\R\C$ and $X$ is the invariants of $Y$). On the other hand, giving a unitary $\Z_2$-module is equivalent to giving a complex Hilbert space carrying a unitary involution.
These facts can be encoded by the following pullback diagram:
$$\xymatrix{
  \Hilb_\R\oplus\Rep\Z_2 \ar[r] \ar[d] & \Rep_\R\Z_2 \ar[d] \\
  \Hilb \ar[r] & \Hilb_\R.
}
$$
Delooping for $n$ times then yields the following pullback diagram:
$$\xymatrix{
  (n+1)\Hilb_\R\oplus(n+1)\Rep\Z_2 \ar[r] \ar[d] & (n+1)\Rep_\R\Z_2 \ar[d] \\
  (n+1)\Hilb \ar[r] & (n+1)\Hilb_\R.
}
$$
That is, giving an orthogonal $n$-category $\CX$ is equivalent to giving a unitary $n$-category $\CY$ carrying a (homotopic) anti-unitary involution (again, $\CY=\CX^\C$). It follows that giving an orthogonal $E_m$-multi-fusion $n$-category $\CX$ is equivalent to giving a unitary $E_m$-multi-fusion $n$-category $n$-category $\CY$ carrying an anti-unitary involution (again, $\CY=\CX^\C$). In other words, giving a real form of $\CY$ is equivalent to giving an anti-unitary involution of $\CY$.
\end{rem}

\begin{rem} \label{rem:real-form2}
More generally, we have the following pullback diagram for a finite group $G$:
$$\xymatrix{
  \bigoplus_\pi (n+1)\Rep G^\pi \ar[r] \ar[d] & (n+1)\Rep_\R G \ar[d] \\
  (n+1)\Hilb \ar[r] & (n+1)\Hilb_\R .
}
$$
That is, giving an object of $(n+1)\Rep G^\pi$ is equivalent to giving a unitary $n$-category carrying a (homotopic) orthogonal $G$-action such that $g\in G$ acts unitarily or anti-unitarily if $\pi(g)=1$ or $\dagger$, respectively. 
\end{rem}

\begin{rem} \label{rem:point-fusion-r*}
Example \ref{exam:point-fusion} and Proposition \ref{prop:nVec-RepG} have the following variants. 

(1) Let $G$ be a finite group. A monoidal functor $\rho:G\to\Fun_\R((n+1)\Hilb,(n+1)\Hilb)$ induces an object $X:B G\to (n+2)\Hilb_\R$ of $(n+2)\Rep_\R G$. According to the orthogonal equivalence $$\Fun_\R((n+1)\Hilb,(n+1)\Hilb) \simeq (n+1)\Hilb\oplus(n+1)\Hilb,$$ $\rho$ induces a group homomorphism $\pi:G\to\Z_2^T$ and, in view of Remark \ref{rem:real-form2}, $X$ lifts to an object $\tilde X\in(n+2)\Rep G^\pi$. Since the image of $\tilde X$ in $(n+2)\Hilb$ is $(n+1)\Hilb$, $\tilde X$ is invertible. Moreover, $\bigoplus_{g\in G}\rho(g)$ carries the structure of an orthogonal multi-fusion $n$-category.

(2) Fix a group homomorphism $\pi:G\to\Z_2^T$ and use $U_T(1)$ to denote the group $U(1)$ equipped with the evident $\Z_2^T$-action. Let $B G\to B^{n+2}U_T(1)\sslash\Z_2^T$ be a functor that represents a cohomology class $\omega\in H^{n+2}(G,U_T(1))$. Applying the above construction to the composition $X:B G\to B^{n+2}U_T(1)\sslash\Z_2^T \subset (n+2)\Hilb_\R$ yields an invertible object of $(n+2)\Rep G^\pi$ as well as an orthogonal multi-fusion $n$-category which we denote by $n\Hilb_{G^\pi}^\omega$. 

(3) Let $n\Hilb_{G^\pi}$ denote the orthogonal multi-fusion $n$-category corresponding to the trivial cohomology class in $H^{n+2}(G,U_T(1))$, for example, $0\Hilb_{G^\pi}=\C[G^\pi]$. Then Remark \ref{rem:real-form2} is reformulated into the following equivalencies
$$\LMod_{n\Hilb_{G^\pi}}((n+1)\Hilb_\R) \simeq (n+1)\Rep G^\pi,$$
$$\Fun_{n\Hilb_{G^\pi}}(n\Hilb,n\Hilb) \simeq \Omega(n+1)\Rep G^\pi.$$
\end{rem}

\begin{defn}
A {\em unitary (resp. orthogonal) modular tensor $n$-category}, $n$UMTC (resp. $n$OMTC) for short, is a nondegenerate unitary (resp. orthogonal) braided fusion $n$-category. 
\end{defn}


\begin{rem} \label{rem:umtc-realform}
Let $\CA$ be an $n$OMTC. The orthogonal symmetric monoidal equivalence $\dagger:n\Hilb\to n\Hilb$ induced by the complex conjugation $\dagger:\C\to\C$ induces an anti-$\C$-linear orthogonal braided monoidal equivalence $\tau:\CA^\C\to\CA^\C$. In the special case where $n=1$, this implies that the central charge of the UMTC $\CA^\C$ is divided by 4. According to Remark \ref{rem:real-form}, a UMTC $\CB$ admits a real form if and only if there is an anti-$\C$-linear orthogonal braided monoidal equivalence $\tau:\CB\to\CB$ such that $\tau^2\simeq\Id_\CB$. Using this criterion and the results from \cite{ACRW16,GNN09}, one verifies that the $so(p)_2$ UMTC for prime $p\equiv1\pmod4$ admits real forms. Details will be given in a forthcoming paper \cite{WZ22}. We also remark that the $so(8)_1$ UMTC\footnote{The $so(8)_1$ UMTC is a modular extension of $\sRep\Z_2^f$ of central charge 4.} 
has an obvious real form.
\end{rem}

\begin{exam} \label{exam:repz2}
There are two real forms of the unitary symmetric fusion 1-category $\Rep\Z_2$: one is $\Rep(\Z_2\times\Z_2^T) \simeq \Rep_\R\Z_2$ and the other is $\Rep\Z_4^T$ where $\Z_4^T$ denotes the group $\Z_4$ equipped with the nontrivial homomorphism onto $\Z_2^T$. Indeed, $\Rep\Z_4^T$ has the form $\Hilb_\R\oplus\Hilb_\BH$ hence consists of two simple objects $\one$ and $e$ obeying the fusion rule $e\otimes e=\one\oplus\one\oplus\one\oplus\one$. 

There are four real forms of the UMTC $\FZ_1(\Rep\Z_2)$: one is $\FZ_1(\Rep_\R\Z_2)$ and the other three have the form $\Hilb_\R\oplus\Hilb_\R\oplus\Hilb_\BH\oplus\Hilb_\BH$. Three of them are non-chiral, but one is chiral. Moreover, there are two modular extensions of $\Rep\Z_4^T$: $\FZ_1(\Rep\Z_4^T)$ and the chiral real form of $\FZ_1(\Rep\Z_2)$. There are four modular extensions of $\sRep(\Z_2^f\times\Z_2^T)$: $\FZ_1(\Rep_\R\Z_2)$, the chiral real form of $\FZ_1(\Rep\Z_2)$ and two real forms of the $so(8)_1$ UMTC.
\end{exam}

\section{Classification of gapped quantum liquids with finite symmetries} \label{sec:class-gql}

In this section, we apply the mathematical results obtained earlier to physical problems, specifically the classification of gapped quantum liquids. In the first subsection, we provide a more precise description and detailed discussion of the classification from \cite{KLWZZ20a, KZ22}. In the second subsection, we extend the classification from \cite{KLWZZ20a, KZ22} to include anti-unitary symmetries. Specifically, we predict multiple time-reversal SPT phases when the spacetime dimension is greater than or equal to three. In the third subsection, we use a categorical approach to derive the crystalline equivalence principle \cite{TE18}: the classification of gapped crystalline quantum liquids is identical to that of gapped quantum liquids with onsite symmetry. More precisely, we show that both types of phases can be characterized by group actions on underlying topological orders.

\subsection{Unitary symmetries} \label{sec:unitary-symm}

Let $\CE$ be a nonzero unitary multi-fusion $n$-category. According to \cite{KLWZZ20a} and \cite[Subsection 5.2]{KZ22}, $n$D gapped quantum liquids\footnote{$n$D represents the spacetime dimension.} with internal symmetry $\CE$ are classified (modulo invertible topological orders) by the equivalence classes of pairs $(\CA,\phi)$ where $\CA$ is a unitary $E_0$-multi-fusion $n$-category and $\phi:\FZ_0(\CE)^\rev \simeq \FZ_0(\CA)$ is a unitary monoidal equivalence. Two pairs $(\CA,\phi)$ and $(\CB,\psi)$ are equivalent if there is a unitary $E_0$-monoidal equivalence $F:\CA\to\CB$ such that $\FZ_0(F)\circ\phi \simeq \psi$.

Consider the composite $*$-functor $\gamma:\CE \to \FZ_0(\CE) \xrightarrow\phi \FZ_0(\CA)^\rev \to \CA^\op$.
\begin{itemize}
\item If $\gamma$ is a $*$-equivalence, then the pair $(\CA,\phi)$ classifies an SPT order.\footnote{An $n$D fermionic invertible topological order is regarded as an $n$D SPT order with internal symmetry $n\sRep\Z_2^f$.} 
The SPT order is trivial if and only if $\phi$ is induced by $\gamma$.
\item If $\gamma$ is not a $*$-equivalence and if $\gamma$ transforms inequivalent objects or morphisms into inequivalent ones, then the pair $(\CA,\phi)$ classifies an SET order.
\item Otherwise, the pair $(\CA,\phi)$ classifies a symmetry breaking order.
\end{itemize}

Two approaches were proposed in \cite{KLWZZ20a} to the classification of gapped quantum liquids: one is via modular extension and the other is via boundary-bulk relation. The $*$-versions of Theorem \ref{thm:mod-ext-e0} and Corollary \ref{cor:mod-ext-e1} show that these two approaches indeed yield the same results.

\begin{exam}
Let $G$ be a finite group. The $*$-version of Theorem \ref{thm:mext-repG} shows that $n$D SPT orders with onsite symmetry $G$ are classified by pointed unitary fusion $(n-2)$-categories with fusion group $G$. By the $*$-version of Example \ref{exam:point-fusion}(2), there is a pointed unitary fusion $(n-2)$-category $(n-2)\Hilb_G^\omega$ for each $\omega\in H^n(G,U(1))$. 
Since there is no nontrivial invertible unitary $(n-2)$-category for $n\le4$, $n$D SPT orders with onsite symmetries are classified by group cohomology.
However, the classification is beyond group cohomology for $n\ge5$. For example, if $\CA$ is a chiral $(n-4)$UMTC such that $\CA^{\boxtimes m}$ is non-chiral, then $\bigoplus_{i=0}^{m-1} \Sigma_*^2\CA^{\boxtimes i}$ carries the structure of a pointed unitary fusion $(n-2)$-category. It seems to us that this classification is more than that given in \cite{Wen15}.

\end{exam}

\begin{rem} \label{rem:class-ito}
The above classification of $n$D SPT orders with symmetry $n\Hilb$ (aka invertible topological orders) is completely trivial: one has $\CA\simeq n\Hilb$. Conversely, if a fusion $(n-1)$-category describes an $n$D invertible topological order, then it is trivial by the $*$-version of Proposition \ref{prop:fus-trivial}. This is why the above classification is up to invertible topological orders. 

Nevertheless, there is still a categorical approach to invertible topological orders: the invertible objects of $n\Hilb$ (i.e. invertible unitary $(n-1)$-categories) classify those $n$D invertible topological orders with gapped boundaries. This classification is in the sense that the invertible objects of $n\Hilb$ classify the invertible $n$D extended TQFT's with target $n\Hilb$ by the cobordism hypothesis. 

Recall that giving an invertible object of $n\Hilb$ is equivalent to giving a Witt class of $(n-3)$UMTC's by the $*$-version of Corollary \ref{cor:witt=morita=inv}. It follows that the 4D invertible topological orders are classified by the Witt group of UMTC's introduced in \cite{DMNO13}. It is known that the invertible 4D TQFT associated to a UMTC \cite{CY93} is realized by the Walker-Wang model \cite{WW12} and only sensitive to the central charge (mod 8) of the UMTC. But an invertible 4D extended TQFT encodes topological excitations on the gapped boundaries of an invertible topological order hence is sensitive to the Witt class of a UMTC.
\end{rem}


\subsection{Anti-unitary symmetries} \label{sec:anti-unitary-symm}

Consider an $n$D time-reversal SPT order. The topological defects of the SPT order form an orthogonal $E_0$-multi-fusion $n$-category $\CA$ over $n\Rep\Z_2^T$. However, we have $n\Rep\Z_2^T \simeq n\Hilb_\R$ by Example \ref{exam:rep-Gpi}(2). This forces $\CA\simeq n\Hilb_\R$. This situation is similar to Remark \ref{rem:class-ito}: the invertible objects of $n\Hilb_\R$ (i.e. invertible orthogonal $(n-1)$-categories) classify $n$D time-reversal SPT orders with gapped boundaries. This classification is in the sense that the invertible objects of $n\Hilb_\R$ classify the invertible $n$D extended TQFT's with target $n\Hilb_\R$ by the cobordism hypothesis.

\begin{exam} \label{exam:spt-Z2T}
Below are examples of time-reversal SPT's in low dimensions:
\begin{itemize}
\item There is only one invertible orthogonal 0-category $\R$. Thus there is only one 1D time-reversal SPT order.
\item According to Example \ref{exam:orth-cat}(3), there are two invertible orthogonal 1-categories $\Hilb_\R$ and $\Hilb_\BH$. Thus there are two 2D time-reversal SPT orders.
\item There is only one invertible unitary 2-category $2\Hilb$ and there is only one real form of $2\Hilb$. Therefore, there is only one 3D time-reversal SPT order.
\item According to Remark \ref{rem:point-fusion-r*}(2), each cohomology class $\omega\in H^n(\Z_2^T,U_T(1))$ induces an invertible object of $n\Hilb_\R$ which is a real form of $(n-1)\Hilb$. Thus $H^n(\Z_2^T,U_T(1))$ classifies $n$D time-reversal SPT orders that arise from a part of real forms of $(n-1)\Hilb$. 
The above examples match \cite{CGLW13}.
\item There are infinitely many inequivalent invertible objects of $4\Hilb_\R$ arising from chiral OMTC's as shown in Remark \ref{rem:umtc-realform}. Therefore, the classification of 4D time-reversal SPT orders is beyond group cohomology. The 4D time-reversal SPT order associated to the obvious real form of the $so(8)_1$ UMTC was studied in \cite{VS13,WS13}, but the others are missing in the physics literature. See \cite{BBJCW20} for a related work on unoriented 4D TQFT's.
\end{itemize}
\end{exam}

\smallskip

In what follows, classification of gapped quantum liquids with unitary symmetries is modulo invertible topological orders and that with anti-unitary symmetries is modulo $\Z_2^T$ SPT orders.

The classification results from \cite{KLWZZ20a} and \cite[Subsection 5.2]{KZ22} can be generalized to anti-unitary symmetries as follows.
\begin{itemize}
\item Let $\CE$ be a nonzero orthogonal multi-fusion $n$-category. The $n$D gapped quantum liquids with internal symmetry $\CE$ are classified (modulo invertible topological orders or $\Z_2^T$ SPT orders) by the equivalence classes of pairs $(\CA,\phi)$ where $\CA$ is an orthogonal $E_0$-multi-fusion $n$-category and $\phi:\FZ_0(\CE)^\rev \simeq \FZ_0(\CA)$ is an orthogonal monoidal equivalence. Two pairs $(\CA,\phi)$ and $(\CB,\psi)$ are equivalent if there is an orthogonal $E_0$-monoidal equivalence $F:\CA\to\CB$ such that $\FZ_0(F)\circ\phi \simeq \psi$.
\end{itemize}


\begin{exam} \label{exam:newspt-Z4T}
According to Example \ref{exam:repz2}, there are two modular extensions of $\Rep\Z_4^T$. On the other hand, $H^3(\Z_4^T,U_T(1))=\Z_1$. Therefore, the classification of 3D $\Z_4^T$ SPT orders is beyond group cohomology.
The nontrivial 3D $\Z_4^T$ SPT order was claimed in \cite[Section VI.B]{Wen17} but the modular data therein are slightly different from ours.
More generally, $\Rep(\Z_m\times\Z_4^T) \simeq \Rep_\R\Z_m\boxtimes_\R\Rep\Z_4^T$ admits at least two evident modular extensions while $H^3(\Z_m\times\Z_4^T,U_T(1))=\Z_1$ for odd $m$. Again the classification of 3D $\Z_m\times\Z_4^T$ SPT orders is beyond group cohomology.
\end{exam}

\begin{exam} \label{exam:newspt-Z2fxZ2T}
According to Example \ref{exam:repz2}, there are four modular extensions of $\sRep(\Z_2^f\times\Z_2^T)$ while the complexifications of the former two give the trivial modular extension of $\sRep\Z_2^f$. Therefore, there are two fermionic 3D time-reversal SPT orders. We failed to find the nontrivial one in the physics literature. C.f. \cite{GW14}.
\end{exam}

\subsection{Crystalline equivalence principle} \label{sec:ext-symm}

Let $G$ be a finite spacetime group that acts on the $n$D spacetime. We have a group homomorphism $\pi:G\to\Z_2^T$ such that $\pi(g)=1$ or $\dagger$ if $g$ preserves or reverses the orientation, respectively. 
The purpose of this subsection is to derive the crystalline equivalence principle \cite{TE18} for finite spacetime groups from the categorical point of view. That is, we show that the classifications of the following two types of phases are the same:
\begin{enumerate}
\renewcommand{\labelenumi}{$(\alph{enumi})$}
\item $n$D gapped crystalline quantum liquids with the spacetime group symmetry $G$;
\item $n$D gapped quantum liquids with the onsite symmetry $G^\pi$.
\end{enumerate}

Consider an $n$D gapped crystalline quantum liquid with the spacetime group symmetry $G$. The presence of the spacetime symmetry makes it hardly to organize the topological defects of the phase to a higher category. Hence neither the approach of modular extension nor boundary-bulk relation works for the classification. But we might characterize this phase by the $G$-action on the underlying $n$D topological order, say, $\CX$. Indeed, the $G$-action assigns to each group element an invertible domain wall between $\CX$ while such data can be approached as follows.

Recall the higher category $\TO^n_\sk$ of potentially anomalous $n$D topological orders, which is equivalent to $\bullet/(n+1)\Hilb$ or its opposite \cite[Subsection 5.4]{KZ22}. Then $\CX$ can be viewed as an object of $\bullet/(n+1)\Hilb$. That is, $\CX$ is a nondegenerate unitary $E_0$-multi-fusion $n$-category and $\Omega\CX$ is a nondegenerate unitary multi-fusion $(n-1)$-category\footnote{This was used in \cite{KWZ15,JF22} as the definition of a topological order.}. 

\smallskip

For simplicity, we consider first the case where $\pi:G\to\Z_2^T$ is trivial, that is, $G$ preserves the orientation of the $n$D spacetime. In this case, the $G$-action on the topological order $\CX$ is simply characterized by a monoidal functor $\alpha: G \to \Omega\CX$ because an invertible domain wall between $\CX$ is nothing but an invertible object of $\Omega\CX$.

On the other hand, we have the following observation for a nonzero unitary multi-fusion $n$-category $\CE$:
\begin{itemize}
\item $n$D gapped quantum liquids with the internal symmetry $\CE$ are classified by the equivalence classes of pairs $(\CM,\rho)$ where $\CM$ is an invertible unitary $n$-category and $\rho:\CE\to\CM$ is a $*$-functor. Two pairs $(\CM,\rho)$ and $(\CN,\sigma)$ are equivalent if $\CM$ and $\CN$ are equivalent over $\CE$.
\end{itemize}
In fact, given a pair $(\CM,\rho)$, we have a unitary $E_0$-multi-fusion $n$-category $\CA = \FZ_0(\CE,\CM)$ as well as a composite unitary monoidal equivalence $\phi: \FZ_0(\CA) \simeq \FZ_0(\CE)^\rev\boxtimes\FZ_0(\CM) \simeq \FZ_0(\CE)^\rev$. Conversely, $(\CM,\rho)$ can be recovered from the pair $(\CA,\phi)$: $\CM = \CA\boxtimes_{\FZ_0(\CE)}\CE$ and $\rho = \one_\CA\boxtimes_{\FZ_0(\CE)}\Id_\CE$ (see the proof of Theorem \ref{thm:mod-ext-e0}). Moreover, the composition $\CE\to\FZ_0(\CE)\xrightarrow\phi\FZ_0(\CA)^\rev\to\CA^\op$ corresponds to $\CE\to\Fun(\CE,\CM)^\op$, $E\mapsto\rho(E^R\otimes-)$.

Then, we have the following observation confirming the equivalence $(a)\Leftrightarrow(b)$:
\begin{itemize}
\item $n$D gapped quantum liquids with the onsite symmetry $G$ are classified by pairs $(\CX,\alpha)$ where $\CX$ is a nondegenerate unitary $E_0$-multi-fusion $n$-category and $\alpha: (n-1)\Hilb_G \to \Omega\CX$ is a unitary monoidal functor (or, equivalently, a monoidal functor $\alpha: G \to \Omega\CX$).
\end{itemize}
If fact, the underlying topological order of $(\CA,\phi)$ is given by the nondegenerate unitary $E_0$-multi-fusion $n$-category $\CX=\CA\boxtimes_{\FZ_0(n\Rep G)}\Sigma_*(n-1)\Hilb_G$ or, equivalently, $\CX=(\CM,\rho((n-1)\Hilb_G))$. Note that $\rho$ induces a unitary monoidal functor $\alpha: (n-1)\Hilb_G \to \Omega\CX$ and one can recover $(\CM,\rho)$ from the pair $(\CX,\alpha)$ as $\rho=\Sigma_*\alpha$.

\begin{rem}
The above arguments can be generalized to show the following result:
\begin{itemize}
\item $n$D gapped quantum liquids with a finite $n$-group symmetry $\CG$ are classified by pairs $(\CX,\alpha)$ where $\CX$ is a nondegenerate unitary $E_0$-multi-fusion $n$-category and $\alpha: (n-1)\Hilb_\CG \to \Omega\CX$ is a unitary monoidal functor (or, equivalently, a monoidal functor $\alpha: \CG \to \Omega\CX$).
\end{itemize}
\end{rem}

Now we turn to the case where $\pi:G\to\Z_2^T$ is nontrivial. In this case, to give an anti-unitary $G^\pi$-action on a topological order $\CX\in\bullet/(n+1)\Hilb$ amounts to lifting $\CX$ to a real form $\tilde\CX\in\bullet/(n+1)\Hilb_\R$ and giving an orthogonal monoidal functor $\alpha:(n-1)\Hilb_{G^\pi}\to\Omega\tilde\CX$. Here, $\tilde\CX$ is pointed by $F(\one_\CX)$ where $F:\CX\to\tilde\CX$ is the forgetful functor induced by $n\Hilb\to n\Hilb_\R$. 

\begin{exam}
(1) In the special case where $\CX$ has the form $(n\Hilb,X)$, i.e. the topological order $\CX$ has a gapped boundary, $(n\Hilb_\R,X)$ is a lift of $(n\Hilb,X)$. Moreover, giving an orthogonal monoidal functor $\alpha: (n-1)\Hilb_{G^\pi} \to \Omega(n\Hilb_\R,X) = \Fun_\R(X,X)$ is equivalent to giving a $G^\pi$-action on the unitary $(n-1)$-category $X$ by Remark \ref{rem:real-form2}. 

(2) There may be more than one lifts of a given $\CX$. For example, $(\Hilb_\BH,\BH)$ is also a lift of $(\Hilb,\C)$.
\end{exam}

As before, we have the following observation for a nonzero orthogonal multi-fusion $n$-category $\CE$:
\begin{itemize}
\item $n$D gapped quantum liquids with the internal symmetry $\CE$ are classified by the equivalence classes of pairs $(\CM,\rho)$ where $\CM$ is an invertible orthogonal $n$-category and $\rho:\CE\to\CM$ is an orthogonal functor. Two pairs $(\CM,\rho)$ and $(\CN,\sigma)$ are equivalent if $\CM$ and $\CN$ are equivalent over $\CE$.
\end{itemize}
Then, we have the following result confirming the equivalence $(a)\Leftrightarrow(b)$:
\begin{itemize}
\item $n$D gapped quantum liquids with the onsite symmetry $G^\pi$ are classified by triples $(\CX,\tilde\CX,\alpha)$ where $\CX$ is a nondegenerate unitary $E_0$-multi-fusion $n$-category, $\tilde\CX\in\bullet/(n+1)\Hilb_\R$ is a lift of $\CX$ and $\alpha: (n-1)\Hilb_{G^\pi} \to \Omega\tilde\CX$ is an orthogonal monoidal functor.
\end{itemize}
In fact, $n\Rep(\Ker\pi)$ is the complexification of $n\Rep G^\pi$ by Proposition \ref{prop:complex-RepG}. In order to get the underlying topological order of $(\CM,\rho)$, we break the symmetry $G^\pi$ to $\Ker\pi$ and obtain the phase $(\CM^\C,\rho^\C)$, then break $\Ker\pi$ and obtain $\CX=(\CM^\C,\rho^\C((n-1)\Hilb_{\Ker\pi}))$. Let $\tilde\CX = (\CM,\rho((n-1)\Hilb_{G^\pi}))$. Then $\tilde\CX$ is a lift of $\CX$ and $\rho$ induces an orthogonal monoidal functor $\alpha: (n-1)\Hilb_{G^\pi} \to \Omega\tilde\CX$. 
Again, $(\CM,\rho)$ can be recovered from $(\tilde\CX,\alpha)$ as $\rho=\Sigma_*\alpha$. 


\end{document}